\documentclass[a4paper,11pt]{article}
\usepackage[ansinew]{inputenc}
\usepackage[spanish,english]{babel}
\usepackage{amsmath,amsfonts,amssymb,amsthm}
\usepackage{graphicx}
\usepackage[usenames,dvipsnames]{color}
\usepackage{color}
\usepackage{lineno}
\usepackage{enumerate}
\usepackage{graphicx}
\usepackage[colorlinks=true,citecolor=black,linkcolor=black,urlcolor=blue]{hyperref}
\usepackage{epsfig, subfigure}
\usepackage{amscd,latexsym}
\usepackage{float}

\theoremstyle{plain}
\newtheorem{theorem}{Theorem}
\newtheorem{lemma}[theorem]{Lemma}

\newtheorem{prop}[theorem]{Proposition}

\newtheorem{remark}[theorem]{Remark}

\newtheorem*{invariant*}{Invariant}

\newcommand\blfootnote[1]{%
  \begingroup
  \renewcommand\thefootnote{}\footnote{#1}%
  \addtocounter{footnote}{-1}%
  \endgroup
}

\date{}

\title{Metric dimension of maximal outerplanar graphs\thanks{An extended abstract of this work has appeared at the 17th Spanish Meeting on Computational Geometry (EGC 2017).}}

\author{M. Claverol\thanks{{\tt merce.claverol@upc.edu}. Universitat Polit\`{e}cnica de Catalunya, Spain.}
\and A. Garc\'\i a\thanks{{\tt olaverri@unizar.es}. IUMA, Universidad de Zaragoza, Spain.}
\and G. Hern\'andez\thanks{{\tt gregorio@fi.upm.es}. Universidad Polit\'ecnica de Madrid, Spain.}
\and C. Hernando \thanks{{\tt carmen.hernando@upc.edu}. Universitat Polit\`{e}cnica de Catalunya, Spain.}
\and M. Maureso\thanks{{\tt montserrat.maureso@upc.edu}. Universitat Polit\`{e}cnica de Catalunya, Spain.}
\and M. Mora \thanks{{\tt merce.mora@upc.edu}. Universitat Polit\`{e}cnica de Catalunya, Spain.}
\and J. Tejel\thanks{{\tt jtejel@unizar.es}. IUMA, Universidad de Zaragoza, Spain.}}

\begin{document}

\maketitle

\blfootnote{\begin{minipage}[l]{0.3\textwidth} \includegraphics[trim=10cm 6cm 10cm 5cm,clip,scale=0.15]{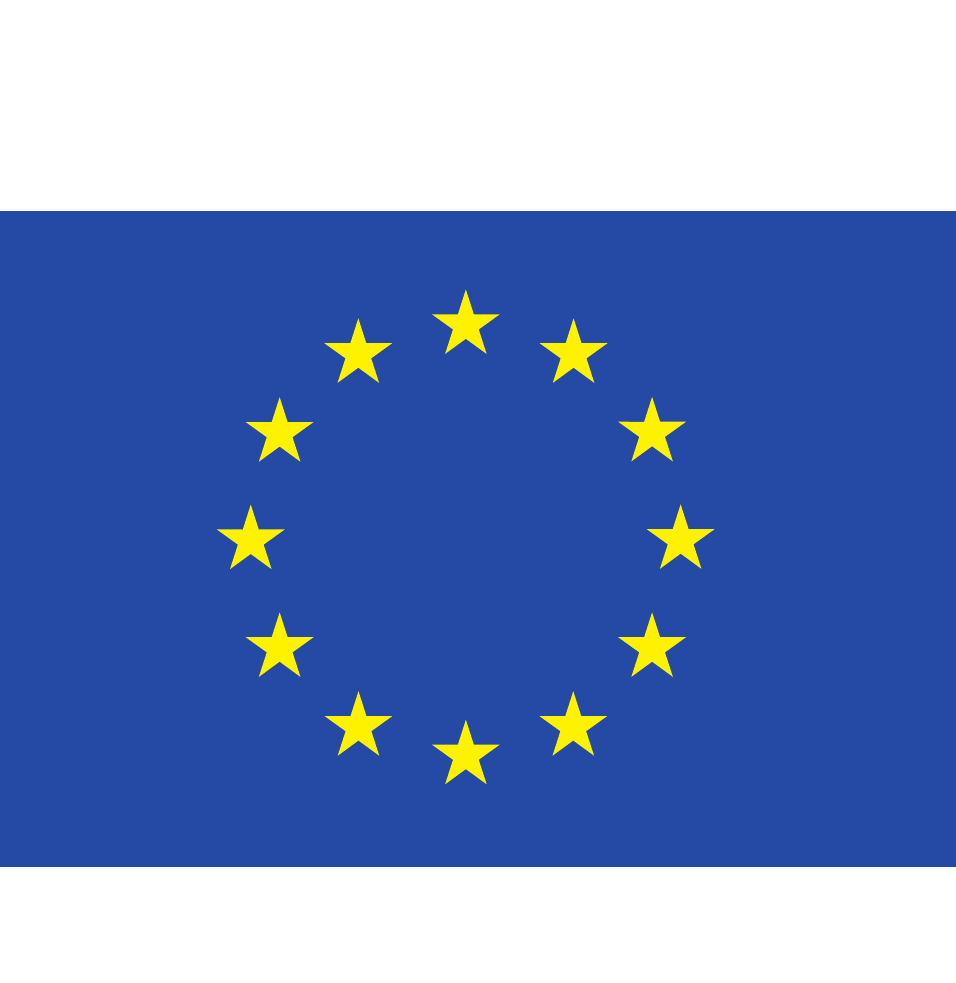} \end{minipage}  \hspace{-2cm} \begin{minipage}[l][1cm]{0.7\textwidth}
 	  This project has received funding from the European Union's Horizon 2020 research and innovation programme under the Marie Sk\l{}odowska-Curie grant agreement No 734922.
 	\end{minipage}}

\begin{abstract}
In this paper we study the metric dimension problem in maximal outerplanar graphs. 
Concretely, if $\beta (G)$ is the metric dimension of a maximal outerplanar graph $G$ of order $n$, we prove that $2\le \beta (G) \le \lceil \frac{2n}{5}\rceil$ and that the bounds are tight. We also provide linear algorithms to decide whether the metric dimension of $G$ is 2  and to build a resolving set $S$ of size $\lceil \frac{2n}{5}\rceil$ for $G$.  Moreover, we characterize the maximal outerplanar graphs with metric dimension 2.

\end{abstract}

\section{Introduction}

Let $G=(V,E)$ be a finite connected simple graph. For two vertices $u,v \in V$, let $d(u,v)$ denote the length of a shortest path in $G$ from $u$ to $v$. If $S=\{ x_1,\dots ,x_k \}$ is a set of vertices of $G$, we denote by $r(u|S)$ the vector of distances from $u$ to the vertices of $S$, that is, $r(u|S)=(d(u,x_1),\dots ,d(u,x_k))$. We say that a vertex $x\in V$ \emph{resolves} a pair of vertices $u,v\in V$ if $d(u,x)\ne d(v,x)$.
A set of vertices $S\subseteq V$ is a \emph{resolving set}
of $G$ if every pair of distinct vertices of $G$ are resolved by some vertex in $S$. Therefore, $S$ is a resolving set if and only if $r(u|S)\not= r(v|S)$ for every pair of distinct vertices $u,v\in V(G)$. The elements of $r(u|S)$ are the \emph{metric coordinates} of $u$ with respect to $S$. A resolving set $S$ of $G$ with minimum cardinality is a \emph{metric basis} of $G$. The \emph{metric dimension} of $G$, denoted by $\beta(G)$,  is the cardinality of a metric basis. The \emph{metric dimension problem} consists of finding a metric basis.

Resolving sets in general graphs were first studied by Slater~\cite{S75} and Harary and Melter~\cite{HM76}. Since then, computing resolving sets and the metric dimension of a graph have been widely studied in the literature due to their applications
in several areas, such as network discovery and verification~\cite{BEEHHMR06}, robot navigation~\cite{KRR96}, chemistry~\cite{CEJO00} or games~\cite{C83}. The reader is referred to~\cite{CHMPP12, CHMPPSW07,  FHHMS15, FMNPV17a, FMNPV17b, GGM-14, GHM-18, GMMRS-14, HMPSW10, MSI-14, YER-17} and the references therein for different results and variants of the metric dimension problem of graphs.

It is well-known that the metric dimension problem in
general graphs is NP-hard~\cite{KRR96}. The problem remains NP-hard even when restricting to some graph classes such as bounded-degree planar graphs~\cite{DPSL17}; split graphs, bipartite graphs and their complements, and line graphs of bipartite graphs~\cite{ELW15}; interval graphs and permutation graphs of diameter 2~\cite{FMNPV17a}. 
Polynomial algorithms
are known for trees~\cite{KRR96};  outerplanar graphs~\cite{DPSL17}; chain graphs~\cite{FHHMS15}; {$k$-edge-augmented trees, cographs and wheels~\cite{ELW15}. 
A weighted variant of the metric dimension problem in several graphs, including paths, trees, and cographs, can be also solved in polynomial time~\cite{ELW15}.

While the algorithms to solve the metric dimension problem in
trees, wheels or chain graphs are linear, the time complexity of the algorithm given in~\cite{DPSL17} for
an outerplanar graph of order $n$ is $O(n^{12})$. Thus, an interesting problem for such graphs is how to find more efficiently a not very large resolving set.
Recall that a graph $G$ is \emph{outerplanar} if it can be drawn in the plane without crossings and with all the vertices belonging to the unbounded face.

In this paper, we focus on studying 
the metric dimension problem in \emph{maximal outerplanar graphs}. A maximal outerplanar graph, \emph{MOP graph} for short, is an outerplanar graph such that the addition of an edge produces a non outerplanar graph. In particular, given a MOP graph $G$ of order $n\ge 3$ we show that $2 \le \beta(G) \le \lceil \frac{2n}{5}\rceil$ and that the bounds are tight. The lower bound is shown to be tight in Section~\ref{sec:md}. Moreover, all MOP graphs with metric dimension 2 are characterized. We also provide in that section a linear algorithm to decide whether the metric dimension of a MOP graph is 2. The tightness of the upper bound is shown in Section~\ref{subsec:fans} by exhibiting a family of MOP graphs attaining the given bound.
Section~\ref{subsec:bound} is devoted to show that the metric dimension of a MOP graph $G$ is at most $\lceil \frac{2n}{5} \rceil$, by building in linear time a resolving set $S$ for $G$ such that $|S| = \lceil \frac{2n}{5} \rceil$. In~\cite{QS17}, it is conjectured
that $\beta(G) \le \lceil \frac{2n}{5}\rceil$ for a maximal planar graph $G$, hence we are answering in the affirmative this conjecture for the particular case of MOP graphs.
We conclude the paper with some open questions in Section~\ref{sec:con}.

To finish this section, we recall some well-known properties of MOP graphs. A MOP graph $G$ {of order at least 3} is biconnected, Hamiltonian and always admits a plane embedding such that all vertices belong to the unbounded face and every bounded face is a triangle.  Unless otherwise stated, we assume throughout the paper that the MOP graph {has order at least 3} and we are given this plane  embedding of $G$. Thus, $G$ can be seen as a triangulation of a convex polygon. Every edge on the boundary of the unbounded face belongs to only one triangle of $G$ and any other edge (called diagonal) belongs to two triangles. The removal of the endvertices of a diagonal makes the graph to be disconnected. $G$ always has at least 2 vertices of degree 2 and when removing any of them (if $|G| \ge 4$), the resulting graph is a MOP graph. From these properties, it is straightforward to see the following result:

\begin{figure}[htb!]
	\centering
	\includegraphics[scale=0.85,page=1]{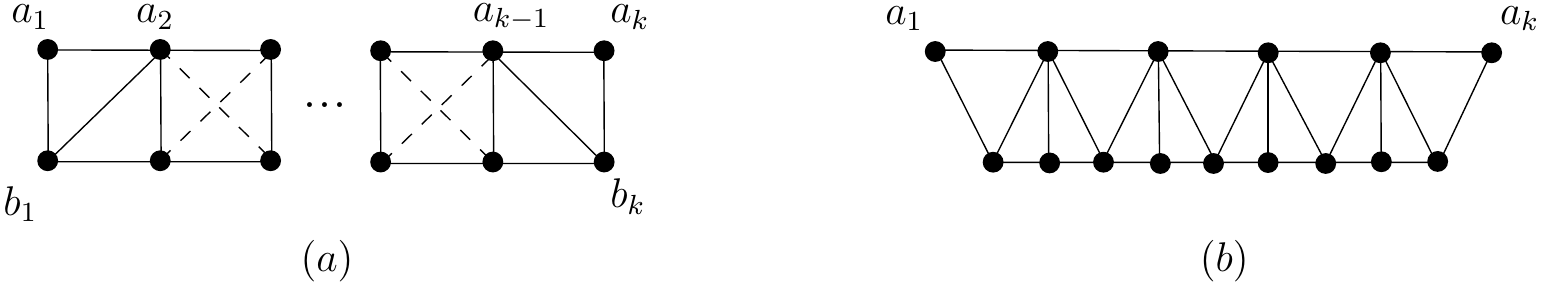}
	\caption{Part (a): Given a 2-tree $G$ with metric dimension 2, a minimal induced 2-connected subgraph containing the basis $\{a_1,a_k\}$, as claimed in~\cite{BDJO17}. Part (b): A 2-tree with metric basis $\{a_1,a_k\}$ whose minimal induced 2-connected subgraph containing $a_1$ and $a_k$ is different from the claimed subgraph in~\cite{BDJO17}.}
	\label{fig:2-tree}
\end{figure}

\begin{remark}\label{diagonales} Let $G$ be a MOP graph and let $xy$ be a diagonal of $G$. If $u$ and $z$ are two vertices belonging to different components of $G\setminus \{x,y\}$, then $d(u,z)\ge \min \{ d(u,x),d(u,y)\} +1$.
\end{remark}

\section{MOP graphs with metric dimension two}\label{sec:md}

Given a MOP graph $G$, its metric dimension must be greater than one, as paths are the only graphs with metric dimension one (see for example~\cite{CEJO00}). In this section, we characterize MOP graphs with metric dimension two.

There are several papers in the literature devoted to study properties of graphs with metric dimension two and to characterize such graphs for certain families of graphs. In~\cite{SH09}, the authors give a general characterization for a graph $G$ to have metric dimension two, based on the distance partition $\{U_1, U_2, \ldots , U_k\}$ of the vertices of $G$, where vertices belonging to $U_i$ are at distance $i$ from a distinguished vertex $v$. They also give a $O(n^2 D^4)$ algorithm to check whether the metric dimension of a graph of order $n$ is two, where $D$ is the diameter of the graph.
In~\cite{KRR96}, the authors show several properties that a graph with metric dimension two must satisfy.

Graphs with metric dimension two have been characterized for some families of graphs. In particular, unicyclic graphs~\cite{DO17} and Cayley graphs~\cite{VBG16}. An incorrect characterization of the 2-trees with metric dimension 2 is given in~\cite{BDJO17}. Starting with a triangle, a 2-tree is formed by repeatedly adding vertices of degree 2 in such a way that each added vertex $u$ is connected to two vertices $v$ and $w$ which are already adjacent. Thus, the family of 2-trees includes MOP graphs as a subfamily.

In~\cite{BDJO17}, the authors define a family $\mathcal{F}$ of 2-trees such that a 2-tree $G$ belongs to $\mathcal{F}$ if $G$ satisfies a set of twelve conditions, and they claim that a 2-tree $G$ has metric dimension 2 if and only if $G$ belongs to $\mathcal{F}$.
When proving that a 2-tree $G$ with metric dimension two must belong to $\mathcal{F}$, the authors claim in one of the cases that the shape of the minimal induced 2-connected subgraph of $G$, containing the two vertices $a_1$ and $a_k$ of the basis of $G$, is as shown in Figure~\ref{fig:2-tree}(a): Two vertices of degree two ($a_1$ and $a_k$), two vertices of degree three ($b_1$ and $b_k$), a set of quadrilaterals with one of the two possible diagonals, and at most one vertex of degree five in the path $a_1, a_2, \ldots , a_k$. But, part (b) of Figure~\ref{fig:2-tree} exhibits a 2-tree $G$ (in fact a MOP graph) with metric dimension two, being $\{a_1,a_k\}$ the only basis of $G$, and the minimal induced 2-connected subgraph of $G$ containing $a_1$ and $a_k$ is precisely $G$, contradicting the shape claimed in~\cite{BDJO17}. As a consequence, their claimed characterization cannot be used to characterize MOP graphs with metric dimension 2.

\begin{figure}[hbt]
\begin{center}	
\includegraphics[width=0.65\textwidth,page=2]{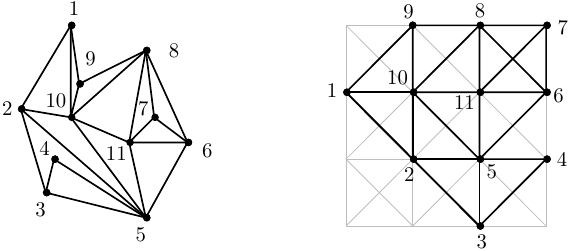}
\caption{Left: A graph $G$ with metric basis $S=\{1,3\}$. The metric coordinates of the vertices are: $r(1|S) = (0,2), r(2|S) = (1,1), r(3|S) = (2,0), r(4|S) = (3,1), r(5|S) = (2,1), r(6|S) = (3,2), r(7|S) = (3,3), r(8|S) = (2,3), r(9|S) = (1,3), r(10|S) = (1,2)$ and $r(11|S) = (2,2)$. Right: The representation $G^*$ of $G$ as a subgraph of $P_n\boxtimes P_n$ with respect to $S$. Vertex $v$ in $G$ is mapped to vertex $v^*$ in $G^*$ such that the cartesian coordinates of $v^*$ are the metric coordinates of $v$.}
\label{fig:ejemplo}
\end{center}
\end{figure}

We next give a characterization for MOP graphs with metric dimension two, based on embedding graphs with metric dimension 2 into the strong product of two paths. The strong product of two paths of order $n$, $P_n\boxtimes P_n$, has the cartesian product $[0,n-1]\times [0,n-1]$ as set of vertices and two different vertices $(i,j)$ and $(i',j')$ are adjacent if and only if $|i'-i|\le 1$ and $|j'-j|\le 1$.
The distance between two vertices of this graph is $d((i,j),(i',j'))=\max \{ |i'-i|,|j'-j| \}$.
We will consider the representation of  this graph in the plane identifying vertex $(i,j)$ with the point with cartesian coordinates $(i,j)$.
In this representation, a path of length $k$ between two vertices $(i,j)$ and $(i',j')$ such that $d((i,j),(i',j'))=k$ is contained in the rectangle having $(i,j)$ and $(i',j')$ as opposite vertices and sides parallel to lines of slope $1$ and $-1$ passing through these vertices.
A set of four vertices of $P_{n}\boxtimes P_{n}$ of the form $ \{(i,j),(i,j+1),(i+1,j+1),(i+1,j) \}$, for some $i,j\in [0,D]$, is a \emph{unit square}.
Three vertices of $P_{n}\boxtimes P_{n}$ are pairwise adjacent if and only if they all belong to a unit square and, in such a case, the edges joining them form a triangle with two consecutive sides one of a unit square and the diagonal joining them.

Let $G$ be a graph with metric dimension $2$ and let $S=\{ u,v \}$ be a metric basis of $G$.
It is straightforward to see that $G$ is isomorphic to a subgraph of the strong product $P_{n}\boxtimes P_{n}$. See Figure~\ref{fig:ejemplo} for an example.
Indeed, we identify vertex $x\in V(G)$ with vertex $(x_1,x_2)\in V(P_{n}\boxtimes P_{n})$, where $(x_1,x_2)=r(x|S)=(d(x,u),d(x,v))$.
Recall that if two vertices $w_1$ and $w_2$ of $G$ are adjacent and $d(w_0,w_1)=d$ for some vertex $w_0$, then $d(w_0,w_2)\in \{ d-1,d,d+1\}$.
Thus, if $x$ and $y$ are adjacent vertices in $G$, then $|d(x,u)-d(y,u)|\le 1$ and $|d(x,v)-d(y,v)|\le 1$, hence $r(x|S)$ and $r(y|S)$ are adjacent in $P_{n}\boxtimes P_{n}$.
We denote by $G^*$ this representation of $G$, that is, $V(G^*)=\{ r(x|S): x\in V(G) \} $ and $r(x|S) r(y|S)\in E(G^*)$ if and only if $xy\in E(G)$. We say that $G^*$  is the \emph{representation of $G$ as a subgraph of $P_{n}\boxtimes P_{n}$ with respect to $S$}, and vertex $(i,j)$ is placed onto the point with cartesian coordinates $(i,j)$.

For every $d\ge 1$, consider the set $A_d=\{ (i,j)\in [0,n-1]\times [0,n-1] :  i+j\ge d, \, \, |j-i|\le d \}$
(see Figure~\ref{fig:prop1} left).
The following properties can be easily derived.

\begin{figure}[hbt]
\begin{center}	
\includegraphics[width=0.7\textwidth,page=3]{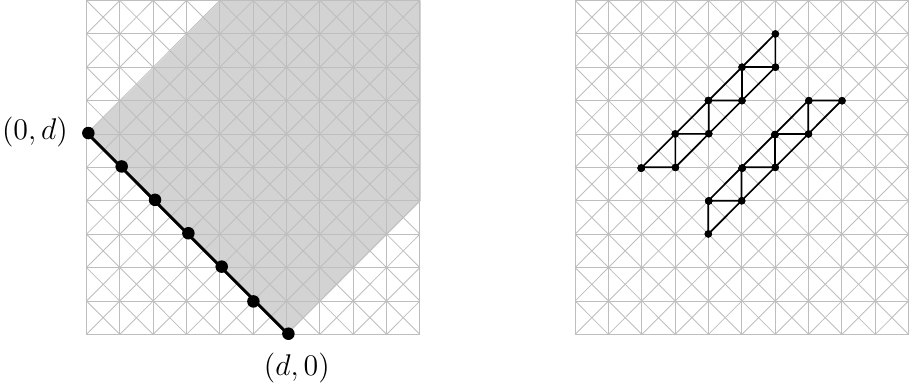}
\caption{Left: Illustrating Proposition~\ref{prop.reprDim2}. The shortest path from $(0,d)$ to $(d,0)$ and the set $A_d$, which is in the shaded region. Right: Examples of horizontal and vertical MOP zigzags.}
\label{fig:prop1}
\end{center}
\end{figure}

\begin{prop}\label{prop.reprDim2}
Let $G$ be a graph with metric dimension $2$,
and let $S=\{u,v\}$ be a metric basis of $G$ such that $d(u,v)=d$.
Consider the representation $G^*$ of $G$ as a subgraph of $P_{n}\boxtimes P_{n}$ with respect to $S$. The following properties hold:
\begin{itemize}
\item[(1)] $S^*=\{(0,d),(d,0)\}$ is a metric basis of $G^*$ and
all the vertices of $G^*$ are in $A_d$.
\item[(2)] There is only one shortest $((0,d),(d,0))$-path in $G^*$, and its vertices are the points $(i,j)$ such that $i+j=d$.
\item[(3)] Three vertices of $G^*$ are pairwise adjacent if and only if they all belong to a unit square.
\end{itemize}
\end{prop}
\begin{proof}

(1) If $d=d(u,v)$, then $r(u|S)=(0,d)$ and $r(v|S)=(d,0)$.  If $x\in V(G)$, then $r(x|S)=(x_1,x_2)=(d(x,u), d(x,v))$. On the one hand,  $x_1=d(x,u)\le d(x,v)+d(u,v)= d(x,v)+d=x_2+d$ and $x_2=d(x,v)\le d(x,u)+d(u,v)=d(x,u)+d=x_1+d$, hence  $|x_1-x_2|\le d$.  On the other hand, $x_1+x_2=d(x,u)+d(x,v)\ge d(u,v)=d$.

(2) There is only one path of length $d$ joining $(0,d)$ and $(d,0)$ in $P_{n}\boxtimes P_{n}$, and its vertices are $\{ (i,j) : i+j=d \}$. Thus, it is also the only shortest path between $(0,d)$ and $(d,0)$ in $G^*$ because we already know that $d_{G^*}((0,d), (d,0))=d$.

(3) It is also obvious, because three pairwise adjacent vertices of $P_{n}\boxtimes P_{n}$ belong to a unit square.
\end{proof}

For $n\ge 5$, we say that a MOP graph $G$ is a \emph{MOP zigzag} if $G$ has two vertices of degree 2, two vertices of degree 3, each one of them adjacent to a different vertex of degree 2, and the rest of the vertices have degree 4. See Figure~\ref{fig:prop1} right for some examples of MOP zigzags. One can see a MOP zigzag as a MOP graph in which the diagonals form a zigzag path connecting the two vertices of degree 3. For $n=3,4$, we consider a triangle and a quadrilateral with a diagonal as MOP zigzags, respectively.

Given the representation $G^*$ of a graph $G$, we say that an edge $e\in E(G^*)$ is \emph{horizontal} if $e=(i,j)(i+1,j)$,  for some $i,j\ge 0$;   \emph{vertical} if $e=(i,j)(i,j+1)$,  for some $i,j\ge 0$;  \emph{$1$-slope diagonal} if $e=(i,j)(i+1,j+1)$,  for some $i,j\ge 0$;  and \emph{$(-1)$-slope diagonal} if $e=(i,j+1)(i+1,j)$,  for some $i,j\ge 0$.
A \emph{vertical MOP zigzag} with base line a vertical edge $(i,j)(i,j+1)$ (see Figure~\ref{fig:prop1} right) is a subgraph of the strong product induced by the set of vertices
$\{ (i+k,j+k): 0\le k\le r\}\cup \{(i+k,j+1+k): 0\le k\le s \}$, for some $r\ge 1$ and $s\in \{ r-1, r\}$, and a \emph{horizontal MOP zigzag} with base line a horizontal edge $(i,j)(i+1,j)$ is a subgraph of the strong product induced by the set of vertices
$\{ (i+k,j+k): 0\le k\le r\}\cup \{ (i+1+k,j+k): 0\le k\le s\}$, for some  $r\ge 1$ and  $s\in \{ r-1, r\}$.

For any integer $k\ge 1$, let $V_k =\{ (i,j) : i+j=k  \}$. The following theorem characterizes the MOP graphs with metric dimension 2. Any of these MOP graphs consists of a base graph similar to the one shown in {Figure~\ref{fig:md2}$(c)$} and several MOP zigzags joined to this base graph (see Figure~\ref{fig:md2}$(d)$).

\begin{theorem}\label{thm:caractMOPs}
Let $G$ be a MOP graph. Then, $\beta (G)=2$ if and only if there is a representation $G^*$ of $G$ as a subgraph of the strong product of two paths such that for some $d\ge 1$,
\begin{itemize}

\item[(1)] $V(G^*)\subseteq A_d$,
$V_d\cap A_d\subseteq V(G^*)$, and $E(G^*)$ contains the edges of the shortest path joining $(0,d)$ and $(d,0)$. 

\item[(2)]  $V_{d+1}\cap A_d\subseteq V(G^*)$ and for each $(i,j)\in V_{d+1}\cap A_d$, $E(G^*)$ contains the edges $(i,j)(i-1,j)$ and $(i,j)(i,j-1)$. 

\item[(3)] For every pair of vertices $(i,j+1)$ and $(i+1,j)$ of $V_{d+1}$ with $i,j\ge 1$,
we have either $(i,j+1)(i+1,j)\in E(G^*)$ or $\{ (i,j+1)(i+1,j+1), (i+1,j)(i+1,j+1),(i,j)(i+1,j+1) \} \subseteq E(G^*)$.
Moreover, if $(i,j+1)(i+1,j)\in E(G^*)$ belongs to two triangles of $G^*$, then $(i+1,j+1)\in V(G^*)$ and $\{ (i,j+1)(i+1,j+1), (i+1,j)(i+1,j+1) \} \subseteq E(G^*)$. 

\item[(4)]  Any other vertex or edge of the graph belongs to a vertical or horizontal MOP zigzag with base line the edge $(0,d)(1,d)$, or the edge $(d,0)(d,1)$, or any other edge of $G$ from those described in the preceding items with an endpoint in $V_{d+1}$ and the other in $V_{d+2}$, with the additional condition that two distinct maximal vertical or horizontal MOP zigzags do not share any edge. 
\end{itemize}

\end{theorem}
\begin{proof}
Let us see first that if a MOP graph has metric dimension 2, then it satisfies items (1)-(4).
Item (1) is a consequence of Proposition~\ref{prop.reprDim2}
(see Figure~\ref{fig:md2}$(a)$).

Let us prove now (2).
Recall that every edge of a MOP graph belongs to at least one triangle.
Let $(i,j-1)(i-1,j)$ be an edge of the $((0,d),(d,0))$-path (and thus, $i+j-1=d$).
The only triangle of the strong product with vertices in $A_d$ containing this edge
is that with vertices $(i,j-1)$, $(i-1,j)$ and $(i,j)$. From here, the second item follows
(see Figure~\ref{fig:md2}$(b)$).

\begin{figure}[hbt!]
	\begin{center}
	\includegraphics[width=1\textwidth,page=4]{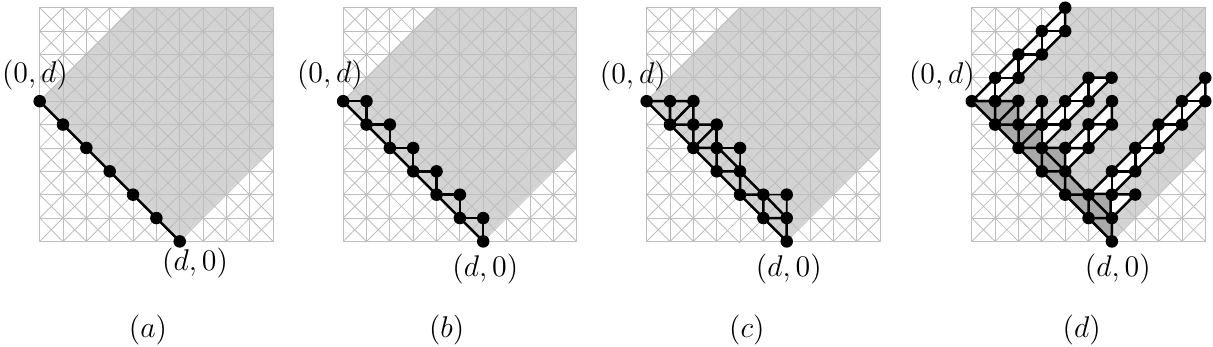}	
	\caption{
	 An example of a MOP graph $G$ with metric dimension 2.  If the vertices of a basis are at distance $d$, then $G$ can be represented as a subgraph $G^*$ of the strong product $P_n\boxtimes P_n$ such that all vertices of $G^*$ belong to the shaded region.
	 Vertices described in Theorem~\ref{thm:caractMOPs} (1), (2), (3) and (4) are added in $(a)$, $(b)$, $(c)$ and $(d)$, respectively. Observe that all vertices of $G^*$ 
belong to the unbounded face.
}\label{fig:md2}
	\end{center}
\end{figure}

To prove item (3), take $(i,j)\in V_d$.
By item (2), we know that  $(i,j)(i,j+1)$ and $(i,j)(i+1,j)$ are edges of $G^*$.
Notice that the edges of the shortest $((0,d),(d,0))$-path belong to exactly one triangle of $G^*$, thus, any other edge incident to $(i,j)\in V_d$ belongs to two triangles of $G^*$. Therefore, the edges $(i,j)(i,j+1)$ and $(i,j)(i+1,j)$ belong to two triangles of $G^*$,
and there are only two possibilities, either $(i,j+1)(i+1,j)\in E(G^*)$ or $\{ (i,j+1)(i+1,j+1), (i+1,j)(i+1,j+1),(i,j)(i+1,j+1) \} \subseteq E(G^*)$. In addition,  if $(i,j+1)(i+1,j)\in E(G^*)$ belongs to two triangles, the only possibility is that $(i+1,j+1)\in V(G^*)$ and $\{ (i,j+1)(i+1,j+1), (i+1,j)(i+1,j+1) \} \subseteq E(G^*)$ (see Figure~\ref{fig:md2}$(c)$).

Finally, let us prove item (4).
Let $t$ be the number of triangles of the MOP graph $G$.
The vertices and edges described in the preceding items (1), (2) and (3) induce a MOP graph, $G_0^*$, with $t_0$ triangles.

If $t=t_0$, then $G_0^*=G^*$ and we are done.
Suppose now that $t > t_0$. In such a case, one of the edges of $G_0^*$ limiting only one triangle in $G_0^*$ must belong to two triangles in $G^*$.
Let $e_0=xy$ be one of these edges and let $z$ be the third vertex of the triangle in $G_0^*$ containing the endpoints of $e_0$.
By definition of $G_0^*$, $e_0$ must be a horizontal edge or a vertical edge.
Besides, $(0,d)$, $(d,0)$ and $z$ belong to the same component in $G^*-\{ x,y \}$.
Assume that $e_0=(i,j)(i+1,j)$ if $e_0$ is a horizontal edge, and $e_0=(i,j)(i,j+1)$ if $e_0$ is a vertical edge, with $i,j\ge 0$. By Remark~\ref{diagonales}, we have that
the third vertex of the other triangle of $G^*$ limited by $e_0$ must be $(i+1,j+1)$.

Let $G_1^*$ be the graph obtained by adding to the graph $G_0^*$ the vertex $(i+1,j+1)$ and the edges joining $(i+1,j+1)$ with the endpoints of $e_0$. Notice that one of the edges added to $G_0^*$ is a 1-slope diagonal edge, and the other one is a horizontal edge if $e_0$ is vertical, or a vertical edge if $e_0$ is horizontal.

Now, if $G^*=G_1^*$, we are done.
Otherwise, there is an edge $e_1$ belonging to exactly one triangle in $G_1^*$ and to two triangles in $G^*$.
By Remark~\ref{diagonales},
there is no $1$-slope diagonal edge $(i,j)(i+1,j+1)$, with $i+j\ge d+1$, limiting two triangles in $G^*$.
Hence, $e_1$ must be a horizontal edge or a vertical edge and we proceed as for $e_0$.
We repeat this procedure until we have added $t-t_0$ triangles to $G_0^*$.
Observe that the new triangles added to $G_0^*$ form a vertical or horizontal MOP zigzag with one of the considered base lines, since the triangles recursively added to $G_0^*$ share vertical or horizontal edges.

Finally, it is not possible that two maximal vertical or horizontal MOP zigzags share an edge $e$. Indeed, in such a case, the edge $e$ should be a $1$-slope diagonal edge $e=(i,j)(i+1,j+1)$, with $i+j\ge d+1$, and $G^*-\{ (i,j) , (i+1,j+1)\}$ would be connected, a contradiction because $e$ is not an edge of the 
unbounded face
(see Figure~\ref{fig:md2}$(d)$).

\begin{figure}[htb]
\begin{center}
\includegraphics[width=0.85\textwidth,page=5]{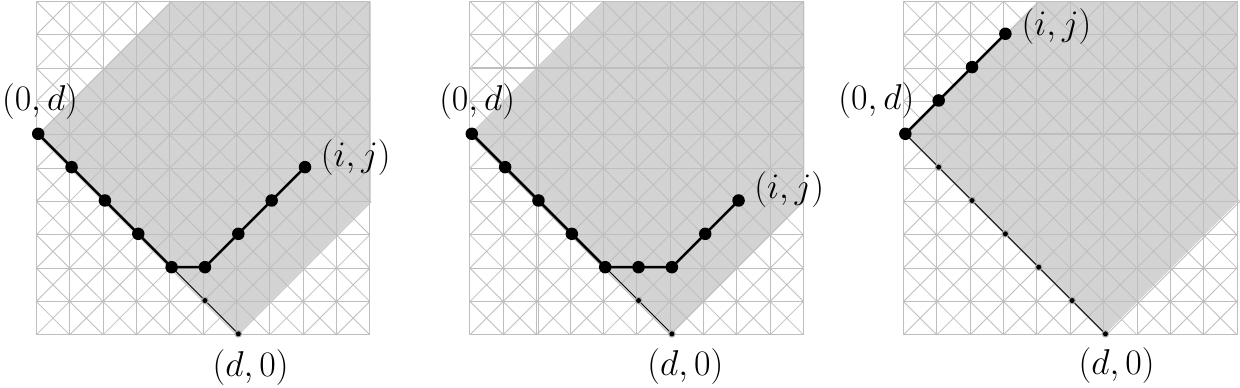}
\caption{A $(0,d)-(i,j)$ path of length $i$ when $i-j$ and $d$ have distinct parity  (left); when $i-j$ and $d$ have the same parity and $i-j\not=-d$  (center) and
 when $i-j=-d$  (right).}\label{fig:caminominimo}
\end{center}
\end{figure}
%

Now, we are going to prove that every graph satisfying (1) to (4) is a MOP graph with metric dimension $2$. By construction, a graph satisfying conditions (1)-(4) is a biconnected plane graph with all vertices belonging to the 
unbounded face and any other face is a triangle. Therefore, $G$ is a MOP graph.
Moreover, $d((0,d),(i,j))=i$ and $d((d,0),(i,j))=j$.
Indeed, it is easy to give a path of length $i$ from $(0,d)$ to $(i,j)$  using some vertices of the shortest $(0,d)-(d,0)$ path;
all the vertices $(i',j')$ such that $i'\le i$, $j'\le j$ and $i'-j'=i-j$; and  vertex
$((d+i-j)/2,1+ (d-(i-j))/2)\in V_{d+1}$, whenever $d$ and $i-j$ have the same parity and with $i-j\not=-d$
(see Figure~\ref{fig:caminominimo}).
In a similar way, a path of length $j$ from $(d,0)$ to $(i,j)$ can be given.
Thus $\{ (0,d),(d,0)  \}$ is a resolving set.  Since $G$ is not a path, we have $\beta (G)=\beta (G^*)=2$.
\end{proof}

{If $G$ is a MOP graph with metric dimension 2, we denote by $G_0^*$ the graph induced by the vertices and edges described in items (1)-(3) of Theorem~\ref{thm:caractMOPs}.}

Deciding whether the metric dimension of a MOP graph is 2 can be done in linear time, as the following theorem shows.

\begin{theorem}\label{thm:algo2}
Given a MOP graph $G$ of order $n$  we can decide in linear time and space whether the metric dimension of $G$ is 2.
\end{theorem}
\begin{proof}
	It is obvious for $n=3$. 	
	From now on,  suppose that $n\ge 4$.
	
    If $G$ is a MOP zigzag, then one can easily check that  
	its metric dimension is 2, {since two of the four vertices of degree 2 and 3 chosen in a suitable way form a resolving set.} Thus, we may assume that $G$ is not a MOP zigzag and we may also assume that the vertices of $G$ are clockwise ordered along its boundary. From Theorem~\ref{thm:caractMOPs}, the representation of a MOP graph $G$ with metric dimension 2 consists of the graph $G^*_0$ together with some vertical and horizontal MOP zigzags joined to $G^*_0$. Note that every vertical or horizontal MOP zigzag finishes in a vertex of degree 2 in $G$.

Given $G$, in the first step of the algorithm we calculate for every vertex $v$ of degree 2 the \emph{maximal MOP zigzag around $v$}, {denoted by $G_v$}. The set of vertices of {$G_v$} is the maximal set of consecutive vertices  
$S_v=\{u, \ldots , v, \ldots , w\}$ of $G$ around $v$ such that the subgraph induced by $S_v$ is a MOP zigzag. 
By definition, { $uw$ is an edge of $G_v$, that will be denoted by $e_v$.} This subgraph can be calculated by alternately exploring the vertices preceding and following $v$ (see Figure~\ref{fig:zigzagproof}). 

\begin{figure}[htb]
\begin{center}
	\includegraphics[width=0.75\textwidth,page=23]{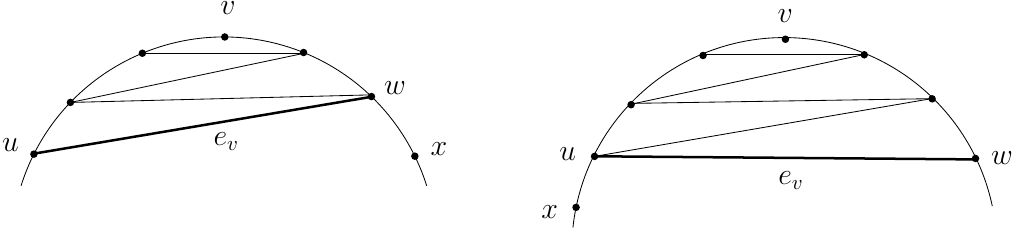}
	\caption{Examples of maximal MOP zigzags around $v$. On the left, vertices $x$ and $u$ are nonadjacent and, on the right, vertices $x$ and $w$ are nonadjacent.}\label{fig:zigzagproof}
\end{center}
\end{figure}

Since calculating $G_v$ only depends on its size, and two maximal MOP zigzags around two vertices $v$ and $v'$ of degree 2 are edge-disjoint, this first step only requires linear time and space. {Notice that if $G$ has metric dimension 2 and $v$ is 
a vertex of degree 2, then the edge $e_v=uw$ of $G_v$} must be a vertical or horizontal edge on the boundary of $G^*_0$, or a $(-1)$-slope diagonal edge of $G^*_0$. This last case (see for example the fourth MOP zigzag when moving along the border of $G_0^*$ from $(0,d)$ in Figure~\ref{fig:md2}$(d)$) only happens when the triangle defined by the vertices $(i,j+1), (i+1,j)$ and $(i+1,j+1)$ belongs to $G^*_0$, with $(i,j+1), (i+1,j)\in V_{d+1}$, and only one of the edges $(i,j+1)(i+1,j+1)$ and $(i+1,j)(i+1,j+1)$ is the base line for a MOP zigzag that contains $v$.

Let $S=\{u_1, u_2, \ldots , u_k\} $ be the set of vertices of degree 2 or 3 of $G$, clockwise ordered when moving along the boundary of $G$. 
From Theorem~\ref{thm:caractMOPs}, we deduce that, {if $G$ has metric dimension 2,} 
then a metric basis of $G$ is formed by two consecutive vertices $u_i$ and $u_{i+1}$ in $S$ (where $u_{k+1}=u_1$).
 Thus, in the second step of the algorithm, 
 {we check  if the set
 $\{ u_i, u_{i+1} \}$ is a metric basis of $G$, for every 
 $i\in \{1,\dots , k\}$. } 

Given a pair $(u_i,u_{i+1})$, this can be done as follows. 
{Suppose that there are $d_i-1$ vertices between $u_i$ and $u_{i+1}$ when traveling clockwise along the boundary of $G$.}
{ Note that using these vertices, checking (and building) if a graph $G^*_0$ as described in items (1), (2) and (3) of Theorem~\ref{thm:caractMOPs} exists can be done in $O(d_i)$ time and space. If such a graph $G^*_0$ exists, the rest of the vertices of $G^*$ must belong to vertical and horizontal MOP zigzags joined to $G^*_0$. 
 This can be again checked in $O(d_i)$ time by visiting clockwise the edges on the boundary of $G^*_0$.
Indeed, an edge $e_v$ associated with a vertex of degree 2 of those calculated in the first step must be a vertical, horizontal or $(-1)$-slope diagonal edge of $G^*_0$. Besides, it can be checked if all vertices of $G$ appear in $G_0^*$ or in a maximal MOP zigzag joined to $G_0^*$, since the number of vertices of $G_0^*$ and of each maximal MOP zigzag is known.
Therefore, as $\sum d_i = n$, this second step also requires linear time and space.}
\end{proof}

\section{Upper bound on the metric dimension of MOP graphs}\label{sec:bound}

In this section, we show that $\beta (G) \le \lceil \frac{2n}{5} \rceil$ for any MOP graph $G$ of order $n$. We also show that, for some special MOP graphs of order $n$, their metric dimension is $\lceil \frac{2(n-2)}{5} \rceil$. Hence, the upper bound $\lceil \frac{2n}{5} \rceil$ is tight when $n$ is a multiple of 5.

In the figures, we will assume   
that the vertices of a MOP graph $G$ are placed on a circle labeled clockwise from $1$ to $n$. The edges will be drawn on or inside the circle as segments or arcs.

\begin{figure}[htb!]
	\begin{center}
		\includegraphics[width=0.25\textwidth,page=6]{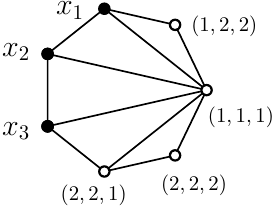}
		\caption{The fan $F_{1,6}$. Black vertices form a metric basis $S$. For a vertex not in $S$, its metric coordinates are given.}
		\label{fig:fan7}
	\end{center}
\end{figure}

\subsection{Fan graphs}\label{subsec:fans}

We first study the metric dimension of a special family of MOP graphs, the fan graphs. A \emph{fan graph} of order $n$, denoted by $F_{1,n-1}$, is a MOP graph such that one of the vertices is connected to the $n-1$ remaining vertices. For $n=3,4,5,6$, one can easily verify that $\beta (F_{1,n-1}) = 2$. For $n=7$, we have $\beta (F_{1,6}) \ge 3$. This result follows from the fact that $n\le \beta +D ^{\beta}$ for a graph with metric dimension $\beta$ and diameter $D$ (see~\cite{KRR96}). As $F_{1,6}$ has diameter 2, if it had metric dimension 2, then $n$ would be at most 6. In addition, the three black vertices of Figure~\ref{fig:fan7} form a metric basis for $F_{1,6}$, so $\beta (F_{1,6}) = 3$.

In the following theorem, we prove that $\beta (F_{1,n-1})= \lceil \frac{2(n-2)}{5} \rceil$, for $n\ge 8$. The proof is based on locating-dominating sets. Given a graph $G=(V,E)$, let $N(u)$ be the set of neighbors of $u$ in $G$, that is, $N(u)=\{ v : uv \in E(G) \}$. A set $S \subseteq V$ is a \emph{dominating set} if every vertex not in $S$ is adjacent to some vertex in $S$. A set $S \subseteq V$ is a \emph{locating-dominating set},
if $S$ is a dominating set and $N(u)\cap   S \neq N(v) \cap S$ for every two different vertices $u$ and $v$ not in $S$. The \emph{location-domination number} of $G$, denoted by $\lambda(G)$, is the minimum cardinality of a locating-dominating set.
It is easy to show that any locating-dominating set is a resolving set. Thus, $\beta (G) \le \lambda (G)$.

\begin{theorem}\label{thm:fan}
Let $n\ge 8$. Then,
{\small
$$\displaystyle{\beta (F_{1,n-1})= \bigg \lceil \frac{2(n-2)}{5} \bigg \rceil} .$$
}
\end{theorem}

\begin{proof}
Observe that $\beta (F_{1,n-1})\ge 3$, because  $F_{1,n-1}$ is not a path and graphs with metric dimension 2 and diameter 2 have order at most $6$. Suppose that  the vertices of $F_{1,n-1}$ are labeled so that
$n$ is the vertex of degree $n-1$, and let $P$ be the path of order $n-1$ induced by vertices from $1$ to $n-1$.

We first prove that $\beta (F_{1,n-1})\le  \big \lceil \frac{2(n-2)}{5} \big \rceil$.
In~\cite{S88}, it is shown that a path of order $n-2$ has a locating-dominating set of size $\big \lceil \frac{2(n-2)}{5} \big \rceil$ such that at least one endpoint of the path does not belong to it. Using this fact, we derive that the path of order $n-2$ induced by the vertices from $2$ to $n-1$ has a locating-dominating set $S$ of size $\big \lceil \frac{2(n-2)}{5} \big \rceil$ such that $2\notin S$.
We claim that $S$ is a resolving set for $F_{1,n-1}$. On the one hand, as $n\ge 8$, then $|S|\ge 3$, so $n$ is the only vertex at distance $1$ from every vertex of $S$. On the other hand, $1$ is the only vertex at distance $2$ from every vertex of $S$, because of the choice of $S$. Finally, every other vertex has different vector of distances to $S$ because their neighborhoods in $S$ are different,
so that the 1's in the vectors of distances to $S$ are located in different places. Consequently, $S$ is a resolving set of $F_{1,n-1}$, and hence $\beta (F_{1,n-1})\le  \big \lceil \frac{2(n-2)}{5} \big \rceil$.

\begin{figure}[ht!]
	\centering
	\includegraphics[width=0.17\textwidth,page=7]{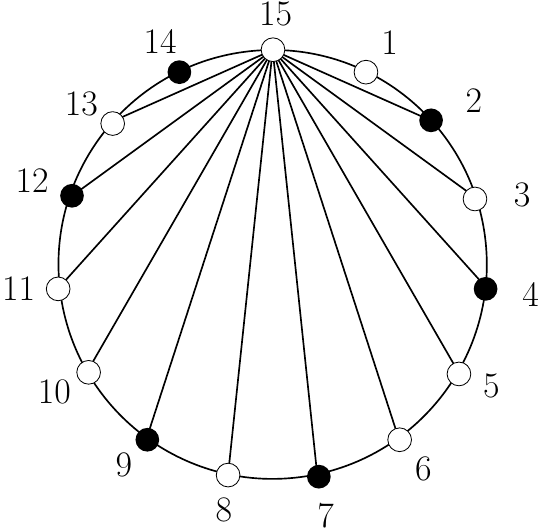}\hspace{2mm}
	\includegraphics[width=0.17\textwidth,page=8]{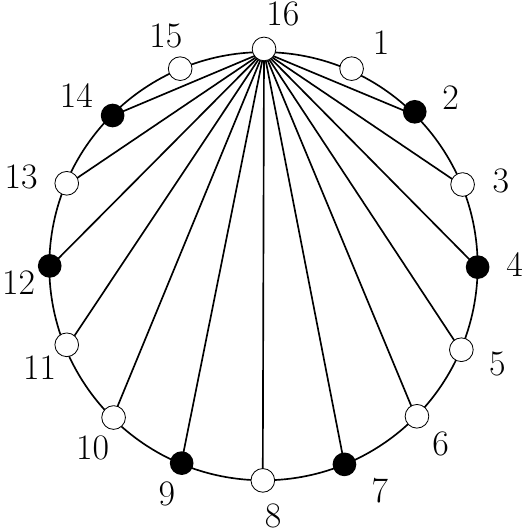}\hspace{2mm}
	\includegraphics[width=0.17\textwidth,page=9]{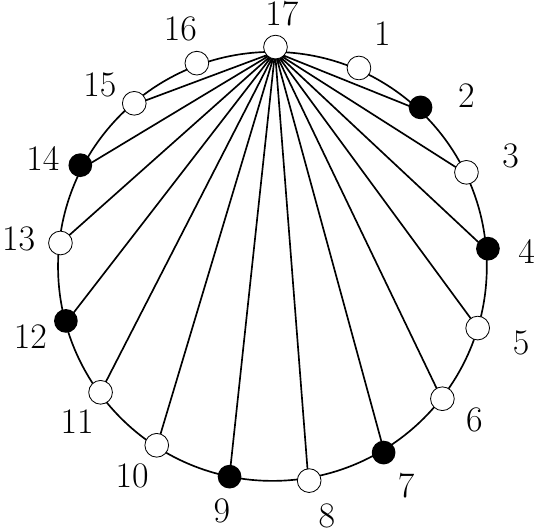}\hspace{2mm}
	\includegraphics[width=0.17\textwidth,page=10]{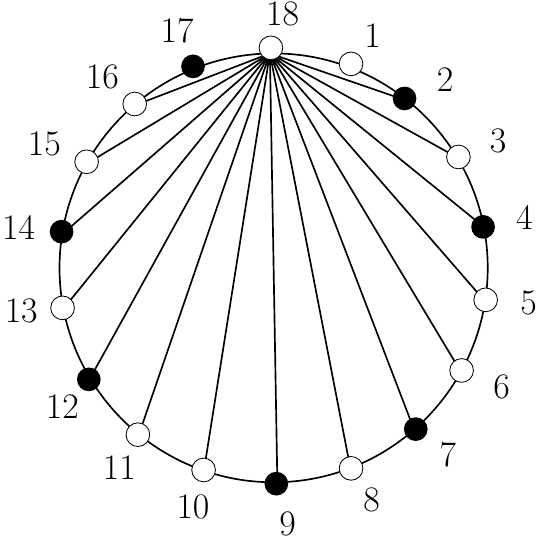}\hspace{2mm}
	\includegraphics[width=0.17\textwidth,page=11]{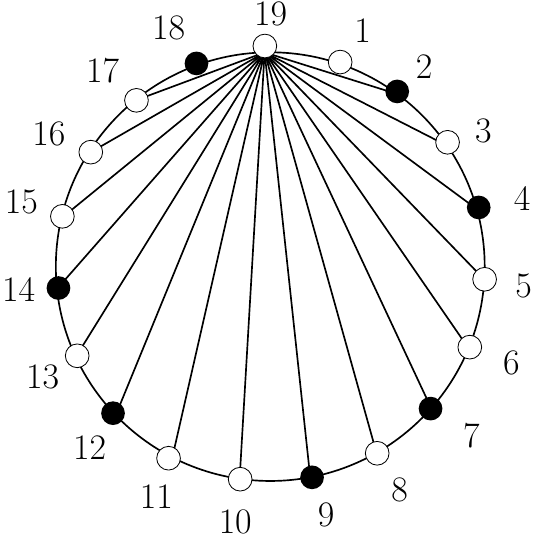}
	\caption{The set of black vertices is a metric basis of the fans of order 15, 16, 17, 18 and 19.}
	\label{fan}
\end{figure}

We now prove that $\beta (F_{1,n-1})\ge  \big \lceil \frac{2(n-2)}{5} \big \rceil$. Let $S$ be a metric basis of $F_{1,n-1}$.
Since $d(i,n)=1$ for $1\le i\le n-1$, vertex $n$ belongs to $S$ only if it has the same coordinates as another vertex $i$ with respect to the set $S\setminus \{ n \}$. Then, $(S\setminus \{ n \})\cup \{ i\}$ is also a metric basis of $F_{1,n-1}$. Hence, we may assume that $n\notin S$ and $n$ is the only vertex with all metric coordinates 1, because $\beta (F_{1,n-1})\ge 3$. Since $F_{1,n-1}$ has diameter 2, all metric coordinates of vertices not in $S$ are 1 or 2. There is at most one vertex with all metric coordinates 2. If there is no vertex with all metric coordinates 2, then $S$ is also a locating-dominating set of the path $P$ of order $n-1$. Hence, $\beta (F_{1,n-1})\ge \lambda (P_{n-1})=\big \lceil \frac{2(n-1)}{5} \big \rceil \ge \big \lceil \frac{2(n-2)}{5} \big \rceil$. If there is one vertex $i_0$ with all metric coordinates 2, then $S$ must be a locating-dominating set for $P-i_0$.
If $i_0\in \{ 1, n-1\}$, then $P-i_o$ is a path of order $n-2$ and $\beta (F_{1,n-1})\ge \lambda (P_{n-2})= \big \lceil \frac{2(n-2)}{5} \big \rceil$.
If $i_0\in \{ 3,\dots ,n-3 \}$, then $P-i_0$ has two connected components that are paths of order $r=i_0-1$ and $s=n-1-i_0$ respectively, with $r+s=n-2$, and we have

$$\beta (F_{1,n-1})\ge
 \lambda (P_{r})+\lambda (P_{s}) = \bigg \lceil \frac{2r}{5} \bigg \rceil+\bigg \lceil  \frac{2s}{5}  \bigg \rceil\ge \bigg \lceil  \frac{2(n-2)}{5} \bigg\rceil.$$

Finally, if $i_0 = 2$, then $1\not\in S$ and $1$ would also be at distance $2$ from every vertex in $S$, a contradiction. Therefore, $i_0 \not= 2$ and, analogously, $i_0\not= n-2$.
\end{proof}

It can be easily verified that 
if $n$ is $5k$, $5k+1$ or $5k+2$ for some $k$, then  $S=\{ 2 + 5r: 0\le r < \lfloor n/5 \rfloor \}\cup \{ 4 + 5r: 0\le r < \lfloor n/5 \rfloor \}$ is a metric basis of  $F_{1,n-1}$,
and if $n$ is $5k+3$ or $5k+4$ for some $k$, then  $S=\{ 2 + 5r: 0\le r < \lfloor n/5 \rfloor \}\cup \{ 4 + 5r: 0\le r < \lfloor n/5 \rfloor \}\cup \{ n-1 \}$ is a metric basis of $F_{1,n-1}$ (see Figure \ref{fan}).
\medskip

\subsection{Upper bound}\label{subsec:bound}

The main goal of this section is to show that
every MOP graph $G=(V,E)$ of order $n$ has  a resolving set $S$ of size $\lceil \frac{2n}{5}\rceil $ that can be built in linear time.
For this purpose, we will begin with a certain set $S$ of vertices of size $ \lceil \frac{2n}{5}\rceil$. If $S$ is a resolving set, we are done.
Otherwise, we will describe how $S$ can be modified to obtain a resolving set  of the same size.
We will refer to the vertices belonging to $S$ as black vertices, and vertices not in $S$ as white vertices.
Recall that the vertices of $G$ are placed on a circle and labeled clockwise from $1$ to $n$, so that all the edges are drawn inside the circle.
A \emph{run} will be a maximal set of consecutive vertices of the same color along the circle.
We will denote by $[i,j]$ the set of vertices $\{i,i+1, \ldots ,j-1,j\}$, if $i<j$, and the set $\{i,i+1,\dots,n,1,\dots ,j-1,j\}$, if $i>j$.

We next prove some technical results.
\begin{lemma}\label{lem:cons}
Let $G$ be a MOP graph of order $n$ and  $i,j\in [1,n]$. If  $i$, $j$, $i-1$ and $j+1$ are four different vertices, then $i$ and $j$ are resolved by either $i-1$ or $j+1$  (mod $n$).
\end{lemma}
\begin{proof}
Observe that $G$ cannot contain at the same time the edges $(i,j+1)$ and $(j,i-1)$ because they cross, whenever $j+1\not= i-1$ (mod $n$). Then, either $i-1$ or $j+1$ resolves $i$ and $j$. See Figure~\ref{fig:alternate}.
 \end{proof}

We have seen in Section~\ref{subsec:fans} that a resolving set of the fan can be obtained with alternating white runs of size 1 and 2 separated by black runs of size $1$.
Such a set is not a resolving set for a general MOP graph $G$,
however, these kinds of sets will play an important role to construct a resolving set of $G$.
This leads us to the following definition.
\vspace{3mm}

\begin{figure}[htb]
\centering
\includegraphics[scale=0.65,page=12]{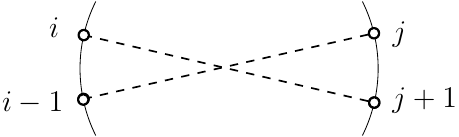}
\caption{Vertices $i$ and $j$ are resolved either by $i-1$ or $j+1$.}\label{fig:alternate}
\end{figure}

We say that an interval $[i,j]$ is \emph{(1,2)-alternating} if and only if
all its white runs have size one or two, black runs have size one and there are no consecutive white runs of the same size.
\vspace{3mm}

Next lemma shows when two white vertices of a (1,2)-alternating  interval
are not resolved by any black vertex of the interval.

\begin{lemma}\label{lem:badcases}
Let $G$ be a MOP graph and let $[i_1,i_2]$ be
a (1,2)-alternating interval
such that the first and last vertices, $i_1$ and $i_2$, are black. Let $S'$ be the set of black vertices in the interval. The following properties hold.
\begin{itemize}
  \item[(1)] Let $i\in [i_1,i_2]$ belong to a white run of size 1. Then, $r(i|S') = r(j|S')$ for some $j\in [i_1,i_2]$ if and only if $j$ belongs to a white run of size 2 and
one of the four cases (a), (b), (c) or (d) of Figure~\ref{fig:onetwo} holds.

  \item[(2)]  Let $i,j\in [i_1,i_2]$ belong to white runs of size 2. Then, $r(i|S') = r(j|S')$  if and only if
 one of the four cases (e), (f), (g) or (h) of Figure~\ref{fig:onetwo} holds.

\item[(3)] If $i\in [i_1,i_2]$ is a white vertex, then there is at most one white vertex $j\in  [i_1,i_2]$ such that $r(i|S')=r(j|S')$.
\end{itemize}
\end{lemma}

\begin{figure}[htb]
\centering
\includegraphics[scale=0.65,page=13]{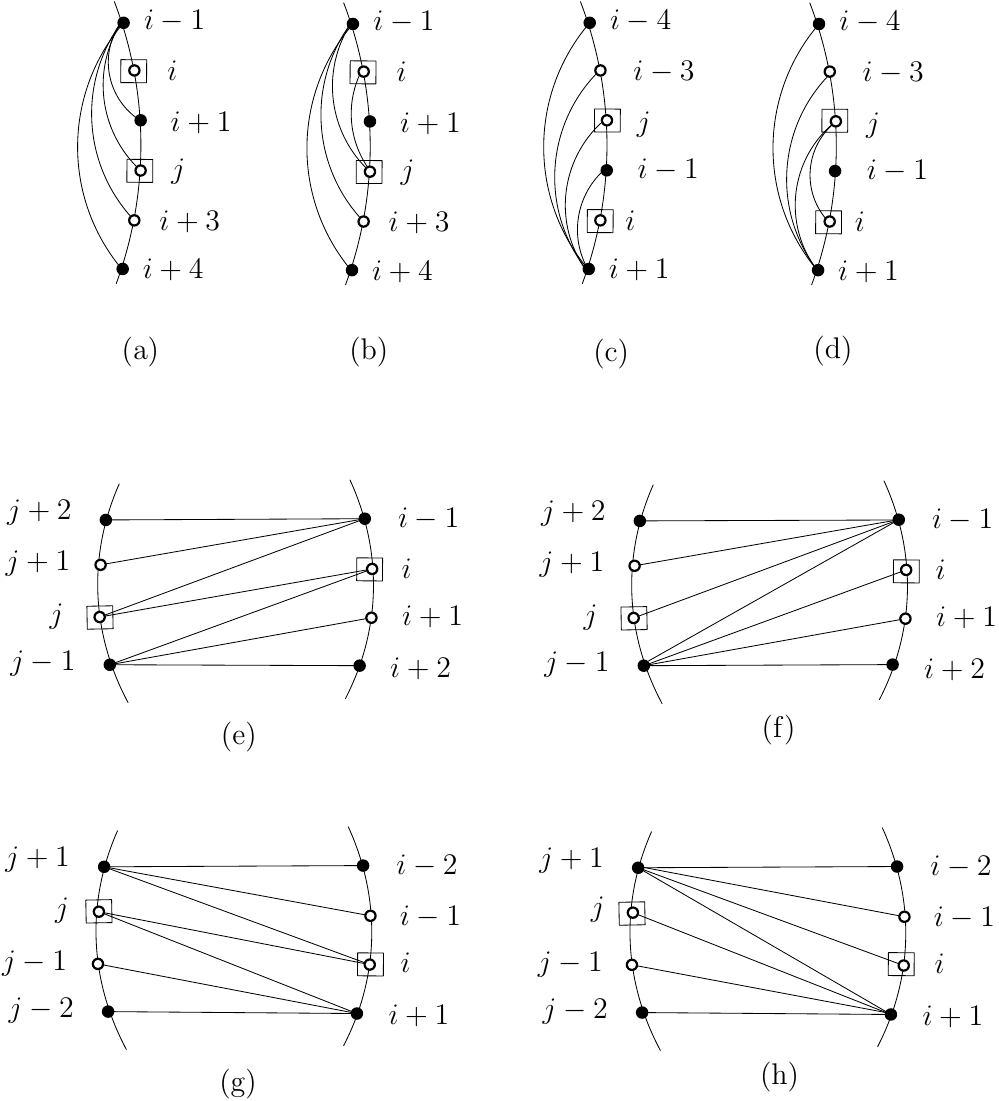}
\caption{The eight cases in which two white vertices $i$ and $j$, squared in the figure, are not resolved by black vertices.}\label{fig:onetwo}
\end{figure}

\begin{proof}
Let us prove item $(1)$.
By Lemma~\ref{lem:cons}, two white vertices belonging to two different runs of size 1 are always resolved by vertices in $S'$.
Suppose now that $i$ belongs to a white run of size 1 and $j$ belongs to a white run of size 2.
If $r(i|S')=r(j|S')$, then, again by Lemma~\ref{lem:cons}, either $j$ is $i+2$ and $j$ is connected to $i-1$, or $j$ is  $i-2$ and $j$ is connected to  $i+1$ (see Figure~\ref{fig:onetwo} top).
Suppose first that $j=i+2$ and $j$ is connected to $i-1$, so that $i-1$ and $i+4$ are black and $i+3$ is white.
Since $j$ is connected to $i-1$, we have that $2\le d(i,i+4)=d(j,i+4)\le 2$. Hence,
vertex $i-1$ is connected to $i+3$ and to $i+4$.  Depending on which edge belongs to $G$, either $(i-1,i+1)$ or $(i,i+2)$, we have Cases (a)--(b) of  Figure~\ref{fig:onetwo}.
Conversely, if Cases (a) or (b) hold, then for any black vertex $i'\in [i_1,i_2]$, the distances from $i'$ to $i$ and $j$ are equal because the shortest path from $i'$ to $i$ or $j$ goes through either $i-1$ or $i+4$, and in these cases the distances from $i$ and $j$ to $i-1$ (resp. to $i+4$) are the same. Thus, $r(i|S') = r(j|S')$.
For the other case, that is, when $j=i-2$ and $j$ is connected to  $i+1$, we have by symmetry that $r(i|S') = r(j|S')$  if and only if
 Cases (c)--(d) of  Figure~\ref{fig:onetwo} hold.
Therefore, we have proved $(1)$.

We now prove item $(2)$. By Lemma~\ref{lem:cons}, it is clear that two white vertices $i$ and $i+1$ belonging to the same white run are resolved by either $i-1$ or $i+2$.
Suppose now that $i$ and $j$ belong to different white runs of size 2 and  $r(i|S')=r(j|S')$.
In such a case, by Lemma~\ref{lem:cons}, if $i-1$ is a black vertex, then $j-1$ is a black vertex, and there is an edge connecting $i$ and $j-1$
and another edge  connecting $j$ and $i-1$.
Also notice that, because of the assumption made in the hypothesis, $i+2$ and $j+2$ are black vertices.
Since $2\le d(i,j+2)=d(j,j+2)\le 2$,
the only possibility for these two distances to be equal is that vertex $i-1$ is connected to both $j+1$ and $j+2$.
Analogously, taking into account that $2\le d(j,i+2)=d(i,i+2)\le 2$, we derive that
vertex $j-1$ must be connected to $i+1$ and $i+2$.
Depending on which edge belongs to $G$, either $(i,j)$ or $(i-1,j-1)$, we have Cases (e)--(f) in Figure~\ref{fig:onetwo}. Conversely, if Cases (e)--(f) hold, then one can easily check that $r(i|S') = r(j|S')$, as there are always  shortest paths from $i$ and $j$ to any other black vertex passing through either $i-1$ or $j-1$.
Cases (g)--(h) of Figure~\ref{fig:onetwo} appear by symmetry when $i+1$ is black instead of $i-1$.
Therefore, $(2)$ follows.
Finally, item (3) is a direct consequence of items (1) and (2).
\end{proof}

Before proving the main result of this section, we give some additional definitions. The vertex of degree two
of Case (c) will be called a \emph{special} vertex.
Let $G=(V,E)$ be a graph.
Given two subsets $S\subset V$ and  $W\subset V$
we say that \emph{$S$ arranges $W$} if every pair of distinct vertices with at least one of them belonging to $W$ is resolved by some vertex in $S$.
If $W$ consists of only one vertex $u$, we say that \emph{$S$ arranges $u$}.
Observe that, by definition, if $S_1$ arranges $W_1$ and $S_2$ arranges $W_2$, then $S_1\cup S_2$ arranges $W_1\cup W_2$. We next prove another technical lemma and the main result of this section, Theorem~\ref{th:bound}.

\begin{lemma}\label{lem:s0}
If $G=(V,E)$ is a MOP graph of order $n$, then there is a set $S_0\subset V=[1,n]$ of black vertices such that $|S_0|= \lceil \frac{2n}{5}\rceil$,
$\{2 \}$ is a white run of size 1, the interval  $[4,n]$ is $(1,2)$-alternating
and $S_0$ arranges the white run $\{ 2 \}$.
\end{lemma}
\begin{proof}
Suppose that $n=5k+t$,  $t\in \{0,1,2,3,4\}$, for some $k\ge 1$.
If $t\in \{0,1,2,3 \}$,
we begin defining $S_0$ as the set of size $\big \lceil \frac{2n}{5} \big \rceil$ that consists of all the vertices of $[1,n]$ of the form $5j+1$ and $5j+3$.
Vertices in $S_0$ are colored black and the rest, white.
If $\{2 \}$ is arranged by $S_0$, we are done.
Otherwise,
by Lemmas~\ref{lem:cons} and \ref{lem:badcases}, $\{2\}$ is the white run of size 1 of some of the subgraphs (a)--(d) of Figure~\ref{fig:onetwo}. If we renumber all the vertices by rotating one place counterclockwise their labels,
and update their colors according to the set $S_0$, the new vertex with label $2$ forms a white run of size 1 not belonging to any of the subgraphs (a)--(d) of Figure~\ref{fig:onetwo}. Hence, $S_0$ arranges $\{ 2 \}$.
If $t=4$, we define $S_0$ as the set formed by vertex $1$ and  all the vertices of $[1,n]$ of the form $5j+3$ and $5j+5$, with $j\ge 0$. Then, $\{ n \}$, $\{ 2 \}$ and $\{4\}$ are three consecutive white runs of size 1 separated by black vertices. By Lemma~\ref{lem:cons}, $S_0$ arranges $2$, and the interval $[4,n]$ is $(1,2)$-alternating.
\end{proof}

\begin{theorem}\label{th:bound}
If $G=(V,E)$ is a MOP graph,  then there exists a resolving set $S\subset V$ such that $|S|= \lceil \frac{2n}{5}\rceil$. Moreover, $S$ can be computed in linear time.
\end{theorem}

\begin{proof}
The general procedure to obtain a resolving set for $G$ is the following. We begin with the set $S:=S_0$ defined in the proof of Lemma~\ref{lem:s0} that arranges the white run $\{2 \}$ of size 1.  If $S_0$ is a resolving set for $G$, we are done. Otherwise, we explore clockwise the white runs of  $S$. Suppose that, after exploring the first $h$ runs, $S$ is a set consisting of $\lceil \frac {2n}5 \rceil$ vertices which is a candidate to be a resolving set for $G$.  If the next white run is not arranged by $S$ because it belongs to some of the subgraphs shown in Figure~\ref{fig:onetwo}, then we define a new set $S'$ by removing some vertices in $S$ and including new ones, such that $|S'| = |S|$ and all explored white runs are arranged by $S'$. Then, we update the set $S$, $S:=S'$, and continue the exploration.

More precisely, suppose that, in a generic step of the exploration, the vertices in the interval $I=[1,i-1]$ have already been explored. Then, we denote by $S$ the set of $\lceil \frac {2n}5 \rceil$ black vertices of $G$ and by $W$ the set of white vertices of $I$, that is, $W=I\setminus S$.

We then prove that $I$ and $S$ satisfy the following invariant.

\begin{invariant*} 
If $I=[1,i-1]$, $S\subseteq V$ and $W=I\setminus S$, then:
\begin{description}
  \item[\it{Property P1.}] $|S|=\lceil \frac {2n}5 \rceil$ and $S$ arranges $W$.
  \item[\it{Property P2.}]  Vertex $i$ is white, vertices $1$ and $i-1$ are black (that is, $\{1,i-1\}\subseteq  S$ and $i \notin S$), and $[i,n]$ is an $(1,2)$-alternating interval.
  \item[\it{Property P3.}] For every white vertex
 $w\in W\setminus \{ i-2  \}$
   and every white {special} vertex $l\in V\setminus I$, there exists a
   black vertex $v\in S\cap (I\setminus \{ i-1\})$ such that $v$ resolves $w$ and $l$.
\end{description}
\end{invariant*}

Obviously, by Property P1, $S$ will be a resolving set  of size $\lceil \frac {2n}5 \rceil$ for $G$ after exploring all white runs.
Properties P2 and  P3 are technical facts that will be needed to proceed with the proof.

We begin verifying that Invariant{} holds for $S=S_0$ and $I=[1,3]$.
By Lemma~\ref{lem:s0}, $S_0$ arranges $2$, and interval $[4,n]$ is $(1,2)$-alternating. Thus, Property P2 obviously holds and Property P3 is true because in this case the set $W\setminus \{ 2 \}= \emptyset$.

Assuming that Invariant{} is true for given sets $I=[1,i-1]$ and $S$, we next show that it holds for new sets $I'$ and $S'$ defined after exploring clockwise the next white run $r$ not belonging to $I$. We will distinguish whether $r$ is already arranged by $S$ or not.

Suppose first that $r$ is arranged by $S$.  If $r$ only consists  of the white vertex $i$,  then one can easily check that Invariant{} holds for $S'=S$ and $I'=I\cup [i,i+1]$. Indeed, as $S$ arranges $W$ and $ i $, then $S$ obviously arranges $W'=W\cup \{i\}$. Property P3 follows from the fact that a vertex in $W$ satisfies Property P3, and a special vertex $l$ cannot be connected to $i-1$, so $i-1$ resolves $i-2$ and $l$. Hence,
$I'$  satisfies Property P3.
If $r$ consists of two white vertices, $i, i+1$,
consider $S'=S$ and $I'=I\cup [i,i+2]$. Sets $S'$ and $I'$  satisfy  Property P1 of Invariant{},
because $S$ arranges $W'=W\cup \{ i,i+1 \}$.
To see that Property P3 is true, notice that a special vertex $l \in V\setminus I'$ is not connected to $i-1$.
Hence, $i-1$ resolves $l$ and any of $i-2$ and $i$.

{Suppose now that $r$ is not arranged by $S$. As $S$ arranges $W$, the white vertices not resolved by $S$ must belong to the interval $[i,n]$. Lemma~\ref{lem:badcases} can be applied to the interval $[i-1,i]$, if $n$ is white, or to the interval $[i-1,n]$, if $n$ is black, since the interval $[i,n]$ is $(1,2)$-alternating by Property 2. Hence, $r$ belongs to one of the subgraphs of Cases (a)-(h). Note that if $r$ has size 1, then $r$ consists of vertex $i$, and if $r$ has size 2, then $r$ consists of vertices $i$ and $i+1$.}

In each one of these 8 cases, the general framework to construct new sets $S'$ and $I'$ satisfying Invariant{} is the following.
{The set $I'$ is obtained by adding an interval $[i,i']$ to $I$, where $i'$ is a black vertex and the vertices of the run $r$ are in $[i,i']$. Then, we interchange the colors of some vertices from $[i-1,i']$, so that the updated set $S'$  of black vertices satisfies $|S'|=|S|=\lceil \frac {2n}5 \rceil$, and Invariant holds for the new sets $I'=I\cup [i,i']$ and $S'$.}
To complete the validity of Property P1, it is needed to show that $S'$ arranges $W'$ after interchanging some colors in $I'$.
This will be proved in two steps.
First, we give a subset of $S'$ that arranges the set of new white vertices, $X= W'\setminus W$.
Secondly, we show that  every pair of white vertices $x$ and $y$, with $x\in W$ and $y\in V\setminus X$, that was resolved  by a vertex from $S\setminus S'$, is now resolved by a vertex from $S'\setminus S$.
Property P2 follows, because we have not changed the colors of the vertices from $V\setminus I'$.
Finally, to prove that the sets $S'$ and $I'$ satisfy Property P3, it is enough to show that it holds for the white vertices in $(X\setminus \{ i'-1\})\cup \{ i-2\}$.
{We next analyze the different cases.}

\vspace{0.5cm}

\noindent

\begin{figure}[htb]
\centering
\includegraphics[scale=0.75,page=14]{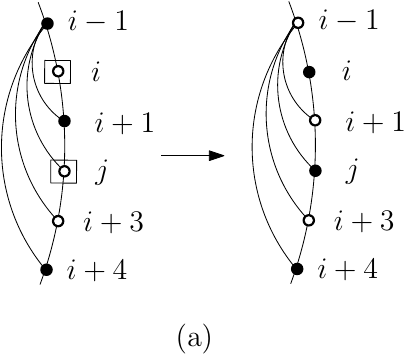}
\hskip 20 mm
\includegraphics[scale=0.75,page=15]{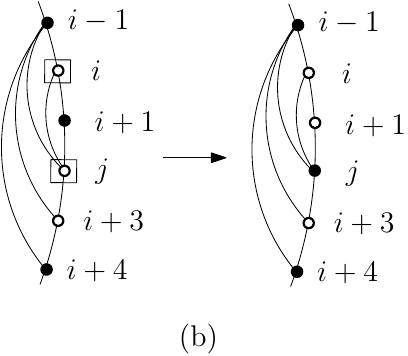}
\caption{ Squared vertices have the same coordinates. Case (a): interchanging the colors of $i$ and $i-1$, and $j=i+2$ and $i+1$. Case (b): interchanging the colors of $i+1$ and $j=i+2$.
}\label{fig:caseab}
\end{figure}

\vspace{0.5cm}

\noindent
{\bf Case (a).}
{The two vertices not resolved by S are i and j as shown in Figure~\ref{fig:caseab} (a).
 In this case, we interchange the colors of vertices $i$ and $i-1$ and the colors of the vertices $j(=i+2)$ and $i+1$.}
We claim that Invariant{} holds for the new sets $I'=I\cup [i,i+4]$ and $S'=(S\setminus \{i-1,i+1\})\cup \{i,i+2\}$.

Let us see that $S'$ arranges $W'=W\cup \{i-1, i+1, i+3\}$.
On the one hand, the set $S^*=\{i,i+2,i+4 \}$ arranges $\{i-1,i+1,i+3\}$, because
$r(i-1|S^*)=(1,1,1)$, $r(i+1|S^*)=(1,1,2)$ and $r(i+3|S^*)=(2,1,1)$, respectively, and the only vertices adjacent to $i+2$ are precisely $i-1$, $i+1$ and $i+3$.
On the other hand, the shortest path from
$i'\in [i+5,i-2]$
to $i$, $i+1$ or $i+2$ necessarily goes through $i-1$,
so $d(i',i) = d(i',i-1)+1$ and $d(i',i+2) = d(i',i+1)$. This implies that, if $i-1$ (resp. $i+1$) resolves two vertices in
$[i+5,i-2]$
then $i$ (resp. $i+2$) also resolves them.
In particular, $S'$ arranges $W'$.

Let us see that Property P3 also holds. Take a special white vertex $l$ in $V\setminus I'$. Property P3 clearly holds for the vertices of $W\setminus \{i-2\}$.
Moreover, since $d(l,i+2)\ge 3$ and the distance from $i+2$ to any of $\{ i-2, i-1, i+1 \}$ is at most two, then
we have that $i+2$ resolves $l$ and any white vertex of $\{ i-2, i-1, i+1 \}$.
Therefore, Property P3 is satisfied, and Invariant{} holds as claimed.

\vspace{0.5cm}

\noindent
{\bf Case (b).}
{
The two vertices not resolved by S are i and j as shown in Figure~\ref{fig:caseab} (b).
In this case, we only need to interchange the colors of vertices $j(=i+2)$ and $i+1$.}
%
We claim  that Invariant{} holds for the new sets $I'=[1,i+4]$ and $S'=(S\setminus \{i+1\})\cup \{i+2\}$. Notice that $W'=W\cup \{i, i+1, i+3\}$.

The set  $S^*=\{i+2,i+4\}$ arranges $\{i, i+1, i+3\}$ because $r(i|S^*)=(1,2)$, $r(i+1|S^*)=(1,3)$ and $r(i+3|S^*)=(1,1)$, and the only white vertices adjacent to $i+2$ are precisely $i$, $i+1$ and $i+3$
(see Figure~\ref{fig:caseab} (b)).
Moreover, observe that for a vertex
$i'\in [i+5,i-2]$,
we have $d(i',i+2) = d(i',i+1)-1$, so if $i+1$ resolves two vertices in
$[i+5,i-2]$, then $i+2$ resolves them  as well. As a consequence, $S'$ arranges $W$.

Finally, to prove Property 3, note that the distance from $i+2$ to a special vertex $l\in V\setminus I'$ is at least 3. Thus, $i+2$ resolves the pairs formed by $l$ and a vertex from
$\{i-2,  i , i+1  \}$.

\begin{figure}[htb]
\centering
\includegraphics[scale=0.7,page=16]{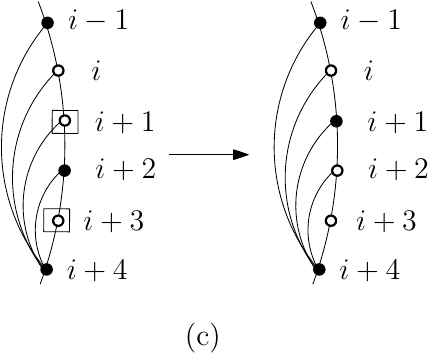}
\hskip 20 mm
\includegraphics[scale=0.7,page=17]{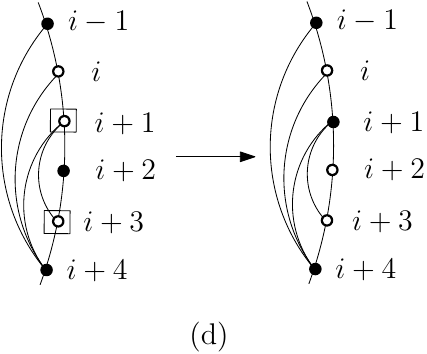}
\caption{Squared vertices have the same coordinates. Case (c): Interchanging the colors of i+1 and i+2, when i+3 is arranged by $(S \setminus  \{i+2\}) \cup \{i+1\}$. Case (d): Interchanging the colors of $i+1$ and $i+2$.}\label{fig:casecd}
\end{figure}

\vspace{0.5cm}

\noindent

\vspace{0.5cm}

\noindent
{\bf Case (c).}
In this case, the vertices not resolved by $S$ are $i+1$ and $i+3$ (see Figure~\ref{fig:casecd} (c)).
We begin by interchanging the colors of the vertices $i+1$ and $i+2$, and distinguish two cases depending on whether $(S\setminus \{i+2\})\cup \{i+1\}$ arranges $i+3$ or not.

Suppose first that $ i+3 $ is arranged by $(S\setminus \{i+2\})\cup \{i+1\}$. Then, Invariant{} holds for the sets $S'=(S\setminus \{i+2\})\cup \{i+1\}$ and $I'=I\cup [i,i+4]$. Note that  $W'=W\cup \{i, i+2, i+3\}$. Indeed, observe that  $S^*=\{i-1,i+1\}$ arranges $\{i,i+2\}$, because $r(i|S^*)=(1,1)$, $r(i+2|S^*)=(2,1)$, and the only white vertices at distance $1$ from $i+1$ are $i$ and $i+2$. Besides,
$d(i',i+2) = d(i',i+1)$ for every vertex
$i'\in [i+5,i-2]$,
implying that
every pair of vertices belonging to $[i+5,i-2]$ that were resolved by $i+2$ are now resolved by $i+1$.
In particular, $S'$ arranges $W$.
Therefore, $S'$ arranges $W'$, and Property P1 holds.
Since a special vertex $l$ in $V\setminus I'$ is not connected to either $i-1$ or $i+1$,
then $l$ together with a vertex from $\{ i-2, i, i+2 \}$
are resolved by either $i-1$ or $i+1$.
Then, Property P3 also holds.

Suppose now that $(S\setminus \{i+2\})\cup \{i+1\}$ does not arrange $i+3$. Let us see which vertex $j$ has the same coordinates as $i+3$ in relation to this set.
Notice that  by Property 3, since $i+3$ is a special vertex,  for any vertex $j\in W\setminus \{ i-2\}$ there is a vertex in $I\cap S$ resolving $i+3$ and $j$. 
{Besides, $i-1$ resolves the pair $i+3$ and $i-2$, and $i+1$ resolves $i+3$ and any of $i$ and $i+2$. Hence, $j \notin [1,i+2]$.}
By Property P2, a vertex $j\in [i+5,n]$ is adjacent to a black vertex $j'$, but $j'$ is not adjacent to $i+3$
unless $j'=i+4$ and $j=i+5$. Then,
$i+6$ is white and $i+7$ is black (see Figure~\ref{fig:casecbis}).
{Since $2 \le d(i+3,i+7)$ and $d(i+5,i+7) \le 2$, we have that both distances are equal only when the edges $(i+4,i+7)$ and $(i+4,i+6)$ belong to G.}

If this situation happens, then $i+3$ and $i+5$ have the same coordinates in relation to $(S\setminus \{i+2\})\cup \{i+1\}$. We remark that Property P3 is important at this time to ensure that $i+5$ is the only vertex with the same coordinates as $i+3$. Otherwise, if Property P3 does not hold, then a vertex $x$ in $W$ could have the same coordinates as $i+3$, because $i+2$ could be the only vertex in $S$ to resolve $i+3$ and $x$.

\begin{figure}[htb]
\centering
\includegraphics[scale=0.7,page=18]{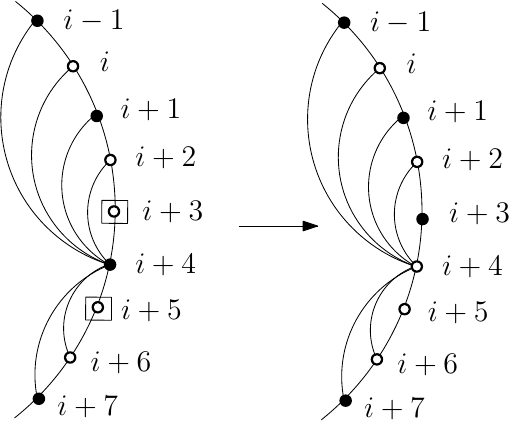}
\caption{Case (c). Squared vertices have the same coordinates. When $ i+3 $ is not arranged by $(S\setminus \{i+2\})\cup \{i+1\}$, the colors of $i+3$ and $i+4$ are interchanged.}\label{fig:casecbis}
\end{figure}

We interchange the colors of vertices $i+3$ and $i+4$, as shown in Figure~\ref{fig:casecbis}, and set $I'=[1,i+7]$ and $S'=(S\setminus \{i+2, i+4\})\cup \{i+1,i+3\}$. Thus, $W'=W\cup \{i, i+2, i+4, i+5, i+6\}$. The argument to prove that Invariant{} holds for these new sets is similar to the previous ones, but a bit more elaborated.

Let us show that $S'$ arranges $W'$.
On the one hand, if $S^*=\{i-1, i+1,i+3\}\subseteq S'$, then $r(i|S^*)=(1,1,2)$, $r(i+2|S^*)=(2,1,1)$ and $r(i+4|S^*)=(1,1,1)$. In addition, the only vertices at distance $1$ from $i+1$ are $i$, $i+2$ and $i+4$. Hence, $S'$ arranges  $\{i, i+2, i+4\}$.
On the other hand, for every vertex $i'\in [i+5,i-2]$, we have $d(i',i+2) = d(i',i+1)$ and $d(i',i+3) = d(i',i+4)+1$. This implies that, any pair of vertices from $[i+5,i-2]$ resolved by $i+2$ or $i+4$ is also resolved by
$i+1$ or $i+3$.
Therefore, since $W\subseteq [i+5,i-2]$ and $S$ arranges $W$, we derive that $S'$ arranges $W$.
It only remains to prove that $S'$ arranges $\{ i+5, i+6 \}$. Notice that $i+7$ resolves the pair $i+5$ and $i+6$.
Moreover, by Property P2, a white vertex $j$ in $V\setminus I'$ is adjacent to a black vertex $j'$. Since $i+5$ and $i+6$ cannot be connected to $j'$, then $j'$ resolves $j$ and any vertex of $i+5$ and $i+6$. Note that if $j=i+8$, then $i+9$ is such a vertex $j'$.
Finally,
as $S'$ arranges $W$ and $\{i, i+2, i+4\}$, a vertex from $\{i+5,i+6\}$ and  a vertex from $[1,i+4]$ are resolved by some vertex of $S'$. Hence, Property P1 is satisfied.

To show that Property P3 holds, we only need to prove this property for the vertices $i-2$, $i$, $i+2$, $i+4$ and $i+5$.
For a special vertex $l$ in $V\setminus I'$, its distance to $i+3$ is at least 3. Since the distance from $i+3$ to $i$, $i+2$, $i+4$, or $i+5$ is at most 2, vertex $i+3$ resolves $l$ and any of these four vertices.
The pair $l$ and $i-2$ is resolved by $i-1$, because $l$ is not adjacent to $i-1$.

\vspace{0.5cm}

\noindent
{\bf Case (d).}
In this case, the vertices not resolved by $S$ are $i+1$ and $i+3$ in the subgraph shown in Figure~\ref{fig:casecd} (d).
This case is symmetric to Case (b). Following the same kind of arguments used in that case, one can easily prove that Invariant{} holds for the sets $I'=[1,i+4]$ and $S'=(S\setminus \{i+2\})\cup \{i+1\}$, defined after interchanging the colors of vertices $i+1$ and $i+2$ (see Figure~\ref{fig:casecd} (d)).

\begin{figure}[htb]
\centering
\includegraphics[scale=0.7,page=19]{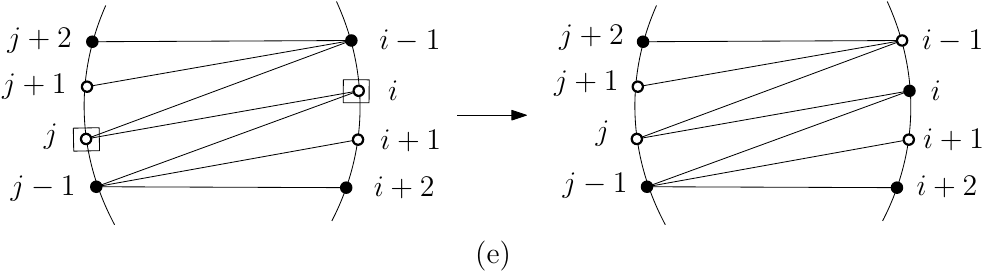}
\hskip 12 mm
\includegraphics[scale=0.7,page=20]{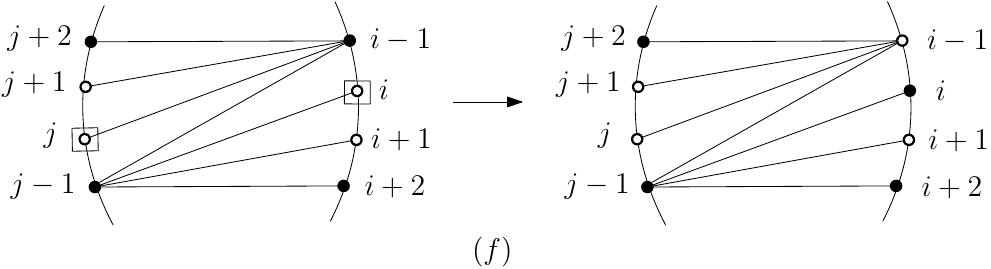}
\caption{Cases (e) and (f): squared vertices have the same coordinates. Interchanging the colors of $i-1$ and $i$.}\label{fig:caseef}
\end{figure}
\vspace{0.5cm}

\noindent
{\bf Case (e).}
In this case, the vertices not resolved by $S$ are $i$ and $j$ in the subgraph shown in Figure~\ref{fig:caseef} (e).
We interchange the colors of vertices $i-1$ and $i$ and we define $I'=I\cup [i,i+2]$ and $S'=(S\setminus \{i-1\})\cup \{i\}$. Thus, $W'=W\cup \{i-1, i+1\}$ (see Figure~\ref{fig:caseef} (e)).

Let us see first that  $S'$ arranges $W'$.
On the one hand, the set $S^*=\{ i,i+2\}$ arranges $\{ i-1, i+1\}$. Indeed, $r(i-1|S^*)=(1,3)$, $r(i+1|S^*)=(1,1)$, and the only white vertices belonging to $V$ at distance 1 from $i$ are $i-1$, $i+1$ and $j$, but $r(j|S^*)=(1,2)$.
On the other hand, $d(i',i)=d(i',i-1)+1$ for every vertex
$i'\in [j+1,i-2]$.
Hence, since
$W\subseteq [j+3,i-2]$,
every pair of vertices with at least one of them belonging to $W$ and the other to
$[j+3,i-2]$
that was resolved by $i-1\in S$ is now resolved by $i\in S'$.
It only remains to prove that every pair formed by a vertex $i'$ from $W$ and a white vertex $j'\in [i+3,j]$ is resolved by some vertex of $S'$. This is true because, by Property P2, $j'$ is adjacent to a black vertex in $S'\cap [i+3,j]$ that is not adjacent to $i'$.
Hence, Property P1 holds.
To prove property P3, it suffices to check that it holds for the vertices $i-1$ and $i-2$. If $l\in V\setminus I'$ is a special vertex different from $j-2$, then $d(l,i)\ge 3$. Hence, the pairs formed by $l$ and a vertex from $\{ i-1, i-2 \}$ are resolved by $i$,  whenever $l\not= j-2$.
Suppose that $l=j-2$ is a special vertex.  If we take a black vertex $j''\ne i-1$ in $I$, then $d(j-2,j'') = 3 + d(i-1,j'')$ and $d(i-2,j'') \le d(i-2,i-1) + d(i-1,j'') = 1 + d(i-1,j'')$. Hence,   $j''\in I'$ resolves $j-2$ and any of $i-1$ and $i-2$.

\vspace{0.5cm}

\noindent
{\bf Case (f).}
In this case, the vertices not resolved by $S$ are $i$ and $j$ in the subgraph shown in Figure~\ref{fig:caseef} (f).
This case is very similar to the previous one.
By interchanging the colors of vertices $i-1$ and $i$ (see Figure~\ref{fig:caseef} (f)), the proof that sets $I'=[1,i+2]$, and $S'=(S\setminus \{i-1\})\cup \{i\}$ satisfy Invariant{} is essentially the same as the proof done in Case (e), with small differences due to the fact that the edge $(i-1,j-1)$ now belongs to $G$ instead of edge $(i,j)$.

\begin{figure}[htb]
\centering
\includegraphics[scale=0.6,page=21]{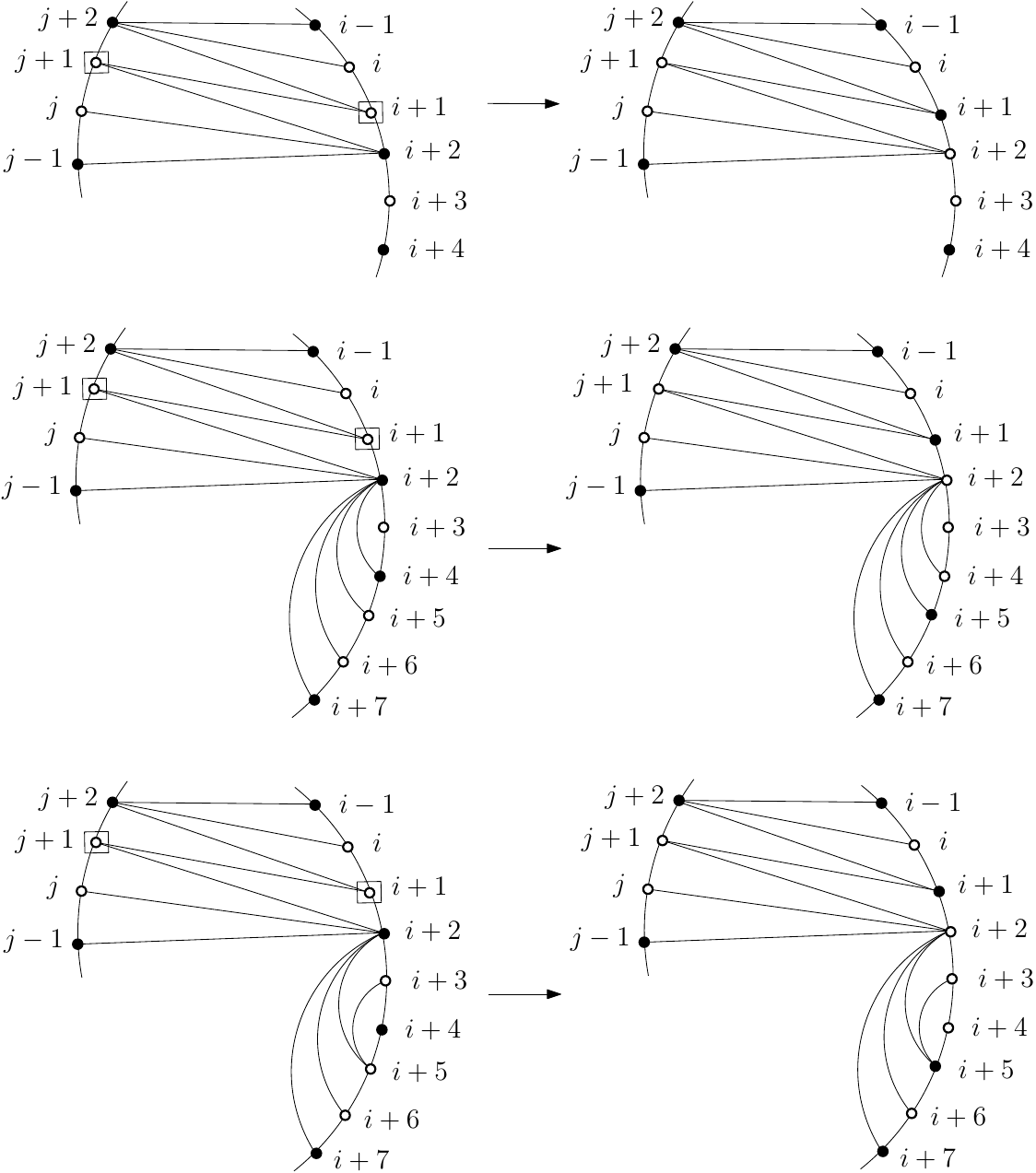}
\caption{Case (g). Squared vertices have the same coordinates. Top: Interchanging the colors of $i+1$ and $i+2$. Middle and bottom: Interchanging the colors of $i+1$ and $i+2$ and the colors of $i+4$ and $i+5$.}\label{fig:caseg}
\end{figure}

\vspace{0.5cm}

\noindent
{\bf Case (g).}
In this case, the vertices not resolved by $S$ are $i+1$ and $j+1$ in the subgraphs shown in Figure~\ref{fig:caseg}.
We begin by interchanging the colors of vertices $i+1$ and $i+2$.
We distinguish two cases depending on whether $(S\setminus \{i+2\})\cup \{i+1\}$ arranges $i+3$ or not.

Suppose first that $(S\setminus \{i+2\})\cup \{i+1\}$ arranges $i+3$  (see Figure~\ref{fig:caseg}, top). We claim that $S'=(S\setminus \{i+2\})\cup \{i+1\}$ and $I'=I\cup [i,i+4]$ satisfy Invariant{}.
Note that $W' = W\cup \{i, i+2, i+3\}$.

On the one hand, the set  $S^*=\{i-1,i+1\}$ arranges $\{i, i+2, j+1\}$.
Indeed, $r(i|S^*)=(1,1)$, $r(i+2|S^*)=(3,1)$, $r(j+1|S^*)=(2,1)$ and the only white vertices adjacent to $i+1$ are $i$, $i+2$ and $j+1$. We include here vertex $j+1$ to ensure that $j+1$ and a vertex in $W$ are resolved.
On the other hand,
a white vertex $i' \in [i+4,j]$ is adjacent by Property P2 to a black vertex $j'$ in this interval. Thus, $j'$
resolves any pair formed by $i'$ together with every white vertex of
$W\subseteq [j+3,i-2]$
because the vertices of this last interval are not adjacent to $j'$.
In addition,
since $d(i',i+1)=d(i',i+2)-1$ for every vertex
$i'\in [j+3,i-1]$
and
$W\subseteq [j+3,i-1]$,
every pair of vertices in this interval, with one of them in $W$, that was resolved by $i+2$ is now resolved by $i+1$.
Hence, $S'$ arranges $W$.
Besides, since a special white vertex $l$ is not connected to either $i+1$ or $i-1$, and vertices $i-2$, $i$ and $i+2$  are adjacent to at least one of them, Property P3 holds.

Suppose now that $(S\setminus \{i+2\})\cup \{i+1\}$  does not arrange $i+3$. Let us see which white vertex $j'$ has the same coordinates as $i+3$ with respect to this set.
A vertex in the interval
$[j,i+1]$
is not adjacent  $i+4$, thus $j'\in [i+2,j-2]$. Moreover, $j'\not= i+2$, because $i+3$ is not adjacent to $i+1$, and consequently,
$j'\in [i+5,j-2]$. Observe now that if $d(i+3, i+1) = d(j',i+1) = 2$ then $d(i+3, i+2) = d(j',i+2) = 1$. Taking vertex $i+2$ as a black vertex, we can apply Lemma~\ref{lem:badcases} to the $(1,2)$-alternating interval $[i+2,n]$ (or $[i+2,1]$), giving rise to the only two possibilities shown in Figure~\ref{fig:caseg}, middle and bottom, for a vertex $j'=i+5$ to have the same coordinates as $i+3$.

Consider the case shown in Figure~\ref{fig:caseg}, middle:  vertex $i+2$ connected to vertices $i+3, i+4, i+5, i+6$ and $i+7$.
In addition to the changes of color of $i+1$ and $i+2$, we interchange the colors of vertices $i+4$ and $i+5$.
We claim that the sets
$S'=(S\setminus \{i+2, i+4\})\cup \{i+1,i+5\}$ and $I'=I\cup [i,i+7]$ satisfy Invariant{}.
Note that $W'=W\cup \{i, i+2, i+3, i+4, i+6\}$.
The set $S^*=\{i-1,i+1,i+5,i+7\}$ arranges $\{ i, i+2,i+4, i+6, j+1 \}$, since  $\{ i, i+2, i+4, i+6, j+1 \}$
are the only white vertices adjacent to $i+1$ or $i+5$,
and $r(i|S^*)=(1,1,3,3)$, $r(i+2|S^*)=(3,1,1,1)$, $r(i+4|S^*)=(4,2,1,2)$, $r(i+6|S^*)=(4,2,1,1)$  and $r(j+1|S^*)=(2,1,2,2)$.
On the other hand,
$i+3$  has no black neighbor. Hence, any other white vertex $i'\in V\setminus W'$ has at least a black neighbor that resolves $i+3$ and $i'$. If $i'\in W\subseteq [j+3,i-2]$, then $i+5$ resolves $i'$ and $i+3$, because $d(i',i+5)\ge 4$, but $d(i+3,i+5)=2$.
Thus, $S'$ arranges $i+3$.
Finally, as $S$ arranges $W$ and
$W\subseteq [j+3,i-2]$,
then $S'$ also arranges $W$ taking into account that
(1)  $j+1$ is already resolved from a vertex in $W$, (2) for every vertex
$i'\in [j+3,i-2]$, we have
$d(i',i+5)=d(i',i+4)$ and $d(i',i+1)=d(i',i+2)-1$
and that (3) a black vertex adjacent to a white vertex of $V\setminus I'$ in the interval $[i+8,j]$ cannot be connected to a vertex in $W$. Therefore, $S'$ arranges $W'$ and Property P1 holds.

To show that Property P3 holds we only need to resolve the pairs formed by a special vertex $l\in V\setminus I'$ and
one of the vertices from $\{ i-2,i,i+2,i+3,i+4\}$ using a vertex from $S'\cap [1,i+5]$. Vertex $i-1$ resolves $l$ and any of $\{i-2, i\}$ because $l$ is not adjacent to $i-1$, and $i+5$ resolves $l$ and any of $\{i+2,i+3,i+4\}$ because $d(l,i+5)\ge 3$. Therefore, Property P3 also holds, and Invariant{} is satisfied, as claimed.

For the last case, the one shown in Figure~\ref{fig:caseg}, bottom, the analysis is very similar to the previous one. Following the same steps as described in the two previous paragraphs, one can prove that $S'=(S\setminus \{i+2, i+4\})\cup \{i+1,i+5\}$ and $I'=I\cup [i,i+7]$ satisfy Invariant{}. The set  $W'$ is $W\cup \{i, i+2, i+3, i+4, i+6\}$.
In this case, it can be shown that the set
 $S^*=\{i-1,i+1,i+5,i+7\}$  arranges $\{ i, i+2, i+3, i+4, i+6, j+1\}$.
Moreover, for every white vertex
$i'\in [j+3,i-2]$,
we have $d(i',i+1) = d(i',i+2)-1$ and  $d(i',i+5) = d(i',i+4)-1$.
Thus, every pair of  white vertices from
$ [j+3,i-2]$
that was resolved by $i+2$ or $i+4$ is resolved now by
$i+1$ or $i+5$.
For every white vertex $i'\in [i+8,j]$, the black vertex adjacent to $i'$ is not adjacent to
a vertex in $W\subseteq [j+3,i-2]$.
Hence, $S'$ arranges $W$.

Finally,  to show that Property P3 holds we only need to resolve the pairs formed by a special vertex $l\in V\setminus I'$ and
one of the vertices from $\{ i-2,i,i+2,i+3,i+4\}$ using a vertex from $S'\cap [1,i+5]$.
All the vertices in $\{ i-2,i,i+2,i+3,i+4\}$  are adjacent to either $i-1$ or $i+5$, but a special vertex $l\in V\setminus I'$ is not adjacent to either $i-1$ or $i+5$, so $i-1$ or $i+5$ resolves $l$ and any of these five vertices. From this, Invariant{} holds as claimed.

\begin{figure}[htb]
\centering
\includegraphics[scale=0.55,page=22]{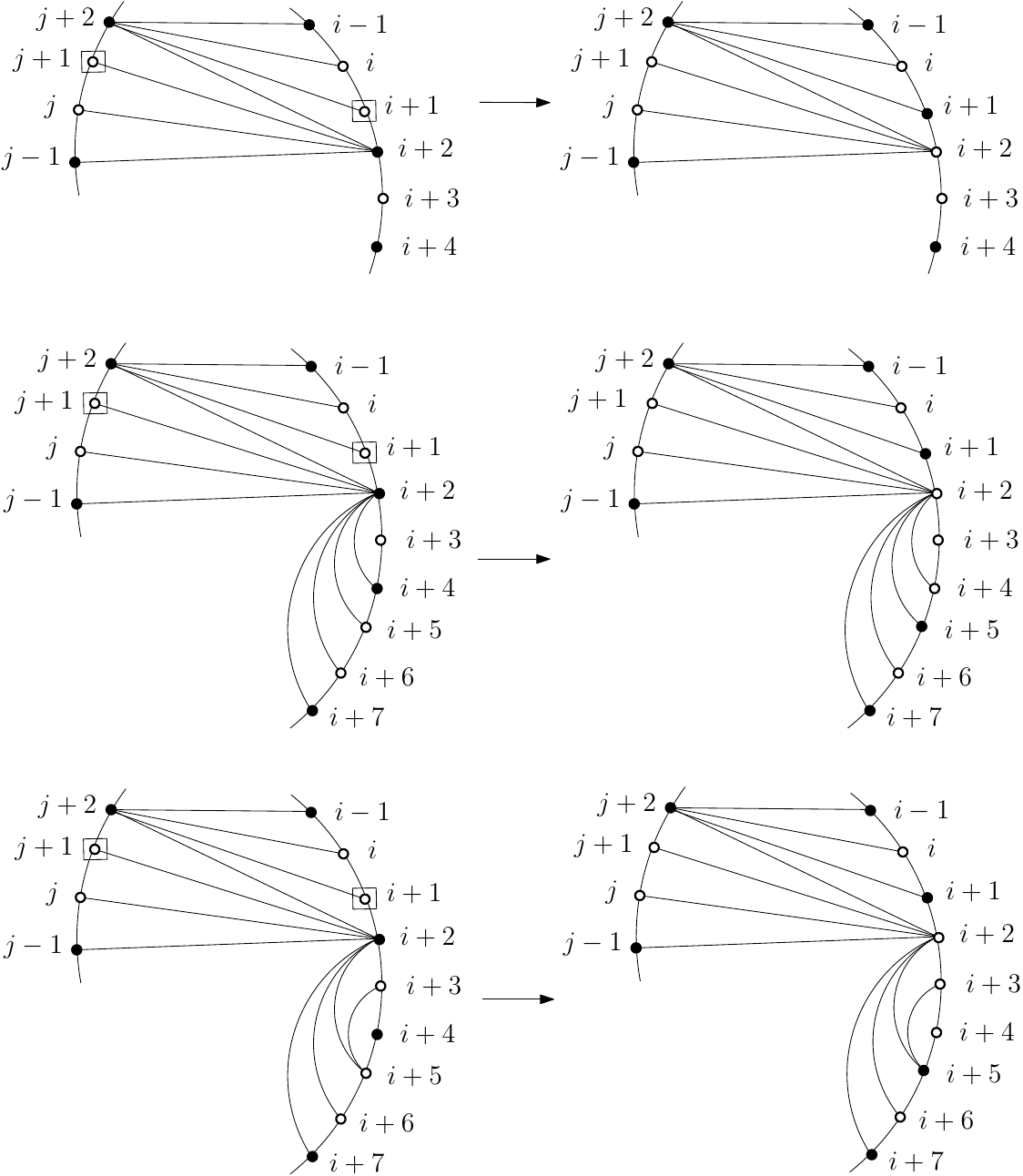}
\caption{Case (h). Squared vertices have the same coordinates. Top: Interchanging the colors of $i+1$ and $i+2$. Middle and bottom: Interchanging the colors of $i+1$ and $i+2$ and the colors of $i+4$ and $i+5$.}\label{fig:caseh}
\end{figure}

\vspace{0.5cm}

\noindent
{\bf Case (h).}
In this case, the vertices not resolved by $S$ are $i+1$ and $j+1$ in the subgraphs shown in Figure~\ref{fig:caseh}.
The analysis of Case (h) follows the same steps as Case (g), although there are small changes due to the fact that now the edge $(i+2,j+2)$ belongs to $G$ instead of the edge $(i+1,j+1)$.

If $(S\setminus \{i+2\})\cup \{i+1\}$ arranges $i+3$, it can be checked that $S'=(S\setminus \{i+2\})\cup \{i+1\}$ and $I'=[1,i+4]$ satisfy Invariant{} (see Figure~\ref{fig:caseh}, top).
If $(S\setminus \{i+2\})\cup \{i+1\}$ does not arrange $i+3$,
arguing exactly as in Case (g), we have that the vertex with the same coordinates as $i+3$ is $j'=i+5$, and one of the cases
shown in  the middle and bottom of Figure~\ref{fig:caseh} holds.
It can be checked in both cases that the sets
$S'=(S\setminus \{i+2, i+4\})\cup \{i+1,i+5\}$ and $I'=I\cup [i,i+7]$ satisfy Invariant{} (see Figure~\ref{fig:caseh}, middle and bottom).
\medskip

To finish the proof of the theorem,
let us see that $S$ can be computed in linear time. Building $S=S_0$ obviously requires linear time. Besides, for every run $r$, we have to check if subgraphs (a)--(h) appear in $G$ and, if it is the case, to update $S$ accordingly. All of this can be done in constant time. Therefore, $S$ can be computed in linear time.
\end{proof}

\section{Conclusions}\label{sec:con}

In this paper, we have studied the metric dimension problem for maximal outerplanar graphs, and we have shown that $2\le \beta (G) \le \lceil \frac{2n}{5}\rceil$ for any maximal outerplanar graph $G$. In relation to the lower bound, we have characterized all maximal outerplanar graphs with metric dimension two, based on embedding such graphs into the strong product of two paths. A first question is whether this technique can be applied to characterize graphs with metric dimension two in other families of graphs, as 2-trees or near-triangulations.

With respect to the upper bound, we have provided a linear algorithm to build a resolving set of size $\lceil \frac{2n}{5}\rceil$ for any maximal outerplanar graph. A second question is whether similar techniques as those described in the algorithm can be used to find efficiently resolving sets for other families of graphs, as Hamiltonian outerplanar graphs or near-triangulations. For near-triangulations, the conjecture is that there always exists a resolving set of size  $\lceil \frac{2n}{5}\rceil$ for any near-triangulation.

\section{Acknowledgments}

A. Garc\'\i a, M. Mora and J. Tejel are supported by H2020-MSCA-RISE project 734922 - CONNECT; M. Claverol, A. Garc\'\i a, G. Hern\'andez, C. Hernando, M. Mora and J. Tejel are supported by project MTM2015-63791-R (MINECO/FEDER);  M. Claverol is supported by project Gen. Cat. DGR 2017SGR1640; C. Hernando, M. Maureso and M. Mora are supported by project Gen. Cat. DGR 2017SGR1336; A. Garc\'\i a and J. Tejel are  supported by project Gobierno de Arag\'on E41-17R.



\begin{thebibliography}{55}


\bibitem{BEEHHMR06}
Z. Beerliova, F. Eberhard, T. Erlebach, A. Hall, M. Hoffman, M. Mihal\'ak and L.S. Ram,
Network discovery and verification,
{\em IEEE J. Sel. Areas Commun.} {\bf 24} (2006), 2168--2181.

\bibitem{BDJO17}
A. Behtoei, A. Davoodi, M. Jannesari and B. Omoomi,
A characterization of some graphs with metric dimension two,
{\em Discrete Mathematics, Algorithms and Applications} {\bf 9} (2017), 175027 (15 pages).

\bibitem{CHMPP12}
J. C\'aceres, C. Hernando, M. Mora, I. Pelayo and M. L. Puertas,   
On the metric dimension of infinite graphs, 
{\em Discrete Appl. Math.} {\bf 160(18)} (2012), 2618--2626. 

\bibitem{CHMPPSW07}
J. C\'aceres, C. Hernando, M. Mora, I.M. Pelayo, M.L. Puertas, C. Seara and D.R. Wood,
On the Metric Dimension of Cartesian Products of Graphs,
{\em SIAM Journal on Discrete Mathematics} {\bf 21} (2007), 423--441.


\bibitem{CEJO00}
G. Chartrand, L. Eroh, M.A. Johnson and O.R. Oellermann,
Resolvability in graphs and the metric dimension of a graph,
{\em Discrete Appl. Math. } {\bf 105} (2000),  99--113.


\bibitem{C83}
V. Chvátal,
Mastermind,
{\em Combinatorica} {\bf 3} (1983),  325--329.

\bibitem{DPSL17}
J. D{\'\i}az, O.  Pottonen, M. Serna and E.J. van Leeuwen,
Complexity of metric dimension on planar graphs,
{\em Journal of Computer and System Sciences} {\bf 83} (2017),  132--158.

\bibitem{DO17}
M. Dudenko and B. Oliynyk,
On unicyclic graphs of metric dimension 2,
{\em Algebra and Discrete Mathematics} {\bf 23} (2017),  216--222.

\bibitem{ELW15}
L. Epstein, A. Levin and G.J. Woeginger,
The (weighted) metric dimension of graphs: hard and easy cases,
{\em Algorithmica} {\bf 72} (2015),  1130--1171.



\bibitem{FHHMS15}
H. Fernau, P. Heggernes, P. van't Hof, D. Meister and R. Saei,
Computing the metric dimension for chain graphs,
{\em Information Processing Letters}, {\bf 115} (2015), 671--676  (2015).


\bibitem{FMNPV17a}
F. Foucaud, G.B. Mertzios, R. Naserasr, A. Parreau and P. Valicov,
Identification, location-domination and metric dimension on interval and permutation graphs. II. Complexity and algorithms,
{\em Algorithmica} {\bf 78} (2017), 914--944.

\bibitem{FMNPV17b}
F. Foucaud, G.B. Mertzios, R. Naserasr, A. Parreau and P. Valicov,
Identification, location-domination and metric dimension on interval and permutation graphs. I. Bounds,
{\em Theoretical Computer Science} {\bf 668} (2017), 43--58.


\bibitem{GGM-14}
D. Garijo, A. González and A. Márquez,
The difference between the metric dimension and the determining number of a graph, {\em Applied Mathematics and Computation}, {\bf 249} (2014), 487-501.

\bibitem{GHM-18}
A. González, C. Hernando and M. Mora,
Metric-locating-dominating sets of graphs for constructing related subsets of vertices,
{\em Applied Mathematics and Computation}, {\bf 332} (2018), 449-456.


\bibitem{GMMRS-14}
C. Grigorious, P. Manuel, M. Miller, B. Rajan and S. Stephen,
On the metric dimension of circulant and Harary graphs,
{\em Applied Mathematics and Computation}, {\bf 248} (2014),  47-54.


\bibitem{HM76}
F. Harary and R.A. Melter,
On the metric dimension of a graph.
{\em Ars Combinatoria}  {\bf 2} (1976), 191-–195.

\bibitem{HMPSW10}
C. Hernando, M. Mora, I.M. Pelayo, C. Seara and D. Wood,
Extremal graph theory for metric dimension and diameter,
{\em Electron. J. Combin.}  {\bf 17} (2010),  R30.


\bibitem{KRR96}
S. Khuller, B. Raghavachari and A. Rosenfeld,
Landmarks in graphs,
{\em Discrete Appl. Math.} {\bf 70} (1996), 217--229.


\bibitem{MSI-14}
H. Muhammad, A. Siddiqui and M. Imran,
Computing the metric dimension of wheel related graphs,
{\em Applied Mathematics and Computation} {\bf 242} (2014), 624-632.



\bibitem{QS17}
C.J. Quines and M. Sun,
Bounds on metric dimension for families of planar graphs,
{\em arXiv:1704.04066v1}.



\bibitem{S75}
P.J. Slater,
Leaves of trees,
{\em Congressus Numerantium} {\bf 14} (1975), 549--559.


\bibitem{S88}
P.J. Slater,
Dominating and reference sets in a graph,
{\em J. Math. Phys. Sci.} {\bf 22}  (1988), 445--455.

\bibitem{SH09}
G. Sudhakara and A.R. Hemanth Kumar,
Graphs with metric dimension two-a characterization,
{\em International Journal of Mathematical and Computational Sciences} {\bf 3} (2009), 1128--1133.


\bibitem{VBG16}
E. Vatandoost, A. Behtoei and Y. Golkhandy Pour,
Cayley graphs with metric dimension two - A characterization,
{\em arXiv: 1609.06565v1}.


\bibitem{YER-17}
I. G. Yero, A. Estrada-Moreno and J.A. Rodríguez-Velázquez,
Computing the k-metric dimension of graphs,
{\em Applied Mathematics and Computation} {\bf 300} (2017),  60-69.


\end{thebibliography}
\end{document}